\documentclass[11pt]{article}

\usepackage[a4paper,margin=1in]{geometry}
\usepackage[T1]{fontenc}
\usepackage[utf8]{inputenc}
\usepackage{lmodern}
\usepackage{xcolor}
\newcommand{\newrev}[1]{{\color{black}#1}}
\usepackage{microtype}
\usepackage{amsmath,amssymb,amsthm,mathtools}
\usepackage{enumitem}
\usepackage[numbers]{natbib}
\usepackage[hidelinks]{hyperref}

\newtheorem{theorem}{Theorem}[section]
\newtheorem{proposition}[theorem]{Proposition}
\newtheorem{corollary}[theorem]{Corollary}
\newtheorem{lemma}[theorem]{Lemma}
\newtheorem{definition}[theorem]{Definition}
\newtheorem{remark}[theorem]{Remark}
\newtheorem{example}[theorem]{Example}

\newcommand{\R}{\mathbb R}
\newcommand{\B}{\mathbb B}
\newcommand{\K}{\mathbb K}
\newcommand{\T}{\mathbb T}
\newcommand{\im}{\operatorname{im}}
\newcommand{\gr}{\operatorname{gr}}
\newcommand{\supp}{\operatorname{supp}}
\newcommand{\qnum}[1]{\left[#1\right]_q}
\newcommand{\A}{\mathcal A}
\newcommand{\Ycal}{\mathcal Y}
\newcommand{\Xcal}{\mathcal X}
\newcommand{\Hcal}{\mathcal H}
\newcommand{\Gcal}{\mathcal G}
\newcommand{\Cop}{\mathcal C}
\newcommand{\Kop}{\mathcal K}
\newcommand{\dY}{\partial^{Y,\mathrm R}_{x,q}}
\newcommand{\dglob}{\partial^{\mathrm R}_{x,q}}
\newcommand{\dscal}{\partial^{\mathrm{sc},Y}_{x,q}}
\newcommand{\dclass}{\partial^{\mathrm R}_x}
\newcommand{\RB}{R_{\B}}

\title{Intrinsic \newrev{Right} \(q\)-Radial Vector Derivatives and Localized Fischer Decompositions on Radial Algebras}

\author{\small Diana Barseghyan (Schneiderov\'a)$^{(1)}$, Juan Bory-Reyes$^{(2)}$, Baruch Schneider$^{(1)}$,
Yifan Zhang$^{(1),(3),(4)}$}
\date{\footnotesize
$^{(1)}$ Department of Mathematics, University of Ostrava,
30. dubna 22, 70103 Ostrava, Czech Republic.\\
E-mail: diana.schneiderova@osu.cz; baruch.schneider@osu.cz; yifan.zhang@osu.cz\\
$^{(2)}$ ESIME-Zacatenco, Instituto Polit\'ecnico Nacional, CDMX 07738, M\'exico.\\
E-mail: juanboryreyes@yahoo.com\\
$^{(3)}$ Department of Algebra, Charles University,
Sokolovsk\'a 83, 18675 Prague, Czech Republic.\\
$^{(4)}$ Department of Applied Mathematics, VSB--Technical University of Ostrava,
17. listopadu 15, 70800 Ostrava, Czech Republic.}

\begin{document}
\maketitle

\begin{abstract}
We define intrinsic right \(q\)-radial vector derivatives on radial algebras of abstract vector variables. For a finite parameter set \(Y\), a scalar Jackson calculus on \(x^2\) and the mixed anticommutators \(\{x,y_i\}\) is combined with the exterior decomposition relative to \(x\). Relabelling covariance and compatibility under inclusions of parameter sets give a direct-limit operator on arbitrary radial algebras. \newrev{Under the dimension specialization \(Q=q^M\), its classical limit agrees in every \(M\)-dimensional Clifford realization with the standard right Clifford derivative; for an exterior blade one obtains
\([x\wedge y_I]\partial_x=(-1)^{|I|}(M-|I|)y_I\).}

Two Fischer-type constructions are considered. The anticommutator with exterior creation is triangular and admits an explicit Green inverse after localization at its diagonal factors. \newrev{For right \(q\)-monogenic elements, the derivative is paired with right multiplication \(R_xF=Fx\). The homogeneous operators \(\partial^{Y,\mathrm R}_{x,q}R_x\) have generically nonzero determinants, which yields determinant-localized Fischer decompositions and explicit projectors.} The one- and two-vector determinants are computed explicitly. \newrev{After \(Q=q^M\), the two-vector determinant is nonzero for every \(0<q<1\) and every positive integer \(M\). For arbitrary finite support, a support filtration factors the determinant into exact-support terms; in degree zero, support rank \(p\) contributes the factor \([m-p]_q+p\).}
\end{abstract}



\section{Introduction}

Radial algebra was introduced by Sommen as an algebra of abstract vector variables whose anticommutators are central \cite{Sommen1997}. It retains the invariant algebraic rules of Clifford vector variables without fixing a coordinate realization, a quadratic form, or a dimension. If only finitely many abstract vectors are involved, a Clifford-polynomial realization is faithful in sufficiently high dimension. The abstract vector derivative then introduces the dimension through the scalar relation \(\partial_x[x]=m\), so the dimension may be treated as a parameter. This framework has been used for symbolic Clifford analysis, Fischer decompositions, Hermitian extensions, superspace, and Dirac complexes \cite{SabadiniStruppaSommenVanLancker2002,DeSchepperGuzmanAdanSommen2017,GuzmanAdan2018}. Standard background on Clifford analysis and the classical Dirac operator can be found in \cite{BrackxDelangheSommen1982}.

Several inequivalent \(q\)-deformations of Clifford analysis are known. Coulembier and Sommen introduced an axiomatic \(q\)-Dirac operator and developed associated operator identities, special functions, and integration \cite{CoulembierSommen2010,CoulembierSommen2011}. A coordinatewise Jackson calculus for commuting variables was studied in \cite{ZimmermannBernsteinSchneider2025}, while a \(q\)-Dirac operator on quantum Euclidean space was constructed in \cite{BernsteinZimmermannSchneider2026}. These theories serve different purposes: the first is based on a deformation of the one-vector Dirac relations, the second uses explicit Jackson partial derivatives, and the third is adapted to a noncommutative quantum space.

The aim of this paper is to construct an explicit operator inside an arbitrary radial algebra that differentiates with respect to one distinguished vector \(x\), allows a finite collection of additional abstract vectors as parameters, and is compatible with adjoining further parameter vectors. Here \emph{intrinsic} means that the construction uses only the radial-algebra structure and the central invariants \(x^2\), \(\{x,y_i\}\), and \(\{y_i,y_j\}\), without choosing coordinates, a Clifford realization, a quadratic form, or a fixed dimension. Compatibility under inclusions of finite parameter sets then defines an operator on the ambient radial algebra by direct limit. This notion is distinct from a coordinatewise Jackson construction; Section~2.2 shows that coordinatewise Jackson differentiation does not preserve the radial subalgebra.

\newrev{Throughout the paper, the classical right Clifford derivative is
\[
 [F]\partial_x:=-\sum_{j=1}^M\frac{\partial F}{\partial x_j}e_j,
 \qquad e_ie_j+e_je_i=-2\delta_{ij}.
\]
If \(A_p\) is an \(x\)-independent exterior \(p\)-vector, then
\[
 [x\wedge A_p]\partial_x=(-1)^p(M-p)A_p.
\]
This identity determines the exterior-degree sign in the definition of the right \(q\)-radial derivative. Section~3 proves the formula directly and records the corresponding low-degree formulas for \(1,x,y,xy,yx,x\wedge y\), and \(\{x,y\}\).}

A coordinatewise Jackson construction already fails in the one-vector radial subalgebra. If \(x=\sum_jx_je_j\) in a Clifford realization and ordinary partial derivatives are replaced coordinatewise by Jackson derivatives, the resulting operator does not preserve \(\R[x^2]\oplus x\R[x^2]\). A radial \(q\)-derivative must therefore dilate the vector variable as a whole and, in the presence of parameter vectors \(y_i\), act on the mixed scalar invariants \(x^2\) and \(\{x,y_i\}\).

For a finite set \(Y=\{y_1,\ldots,y_N\}\), we define a scalar Jackson calculus in the variables
\[
 r=x^2,\qquad s_i=\{x,y_i\},\qquad c_{ij}=\{y_i,y_j\},
\]
with dilation \(r\mapsto q^2r\), \(s_i\mapsto qs_i\), and \(c_{ij}\mapsto c_{ij}\). Combining this scalar calculus with the exterior decomposition relative to \(x\) yields the finite right \(q\)-vector derivative \(\partial^{Y,\mathrm R}_{x,q}\). Relabelling covariance and compatibility under \(Y\subset Z\) give a direct-limit operator \(\partial^{\mathrm R}_{x,q}\) on the scalar extensions of arbitrary radial algebras. We call an element \emph{right \(q\)-monogenic} if it lies in the kernel of the corresponding finite relative operator, or of the direct-limit operator in the global setting.

The first main theorem concerns exterior creation. The anticommutator of the right \(q\)-vector derivative with the operator that adjoins \(x\) exteriorly is triangular, with an explicit diagonal. After central localization at the diagonal factors, the anticommutator has an explicit Green inverse and gives complementary projections.

\newrev{The second main theorem concerns the right \(q\)-monogenic Fischer decomposition. We pair the right derivative with right multiplication \(R_xF=Fx\) and, on each finite homogeneous block, consider \(\Kop^Y_k=\partial^{Y,\mathrm R}_{x,q}R_x\). Its classical limit is \(\partial_x^{\mathrm R}R_x\). An explicit Clifford specialization, together with the standard classical Fischer decomposition, gives generic nonvanishing of the homogeneous determinants. After determinant localization, one obtains the right \(q\)-monogenic direct sum and an explicit projector, compatible with enlargement of the auxiliary set.}

\newrev{The determinants for the smallest supports are computed explicitly. With no auxiliary vector the homogeneous blocks are one-dimensional, while with one auxiliary vector all specialized determinant factors are positive for \(0<q<1\) and every positive integer dimension. For arbitrary support, a filtration reduces the determinant to exact-support factors, whose degree-zero factor at support rank \(p\) is \([m-p]_q+p\).}

\newrev{The paper is organized as follows. Section~2 introduces the radial-algebra setting, the coefficient ring, and the coordinatewise obstruction. Section~3 defines the relative right \(q\)-derivative and establishes its classical limit and basic low-degree formulas. Sections~4 and 5 prove, respectively, the exterior Green decomposition and the right \(q\)-monogenic Fischer decomposition. Section~6 computes the determinants for the smallest supports and studies exceptional numerical specializations.}

\newrev{The case of one auxiliary vector, \(Y=\{y\}\), is used as a running low-dimensional check throughout the paper. Example~\ref{ex:low-degree} tests the defining right-action signs on the elementary products; Example~\ref{ex:H-one-y} follows the same case through the exterior Green operator; Example~\ref{ex:right-pairing} shows why right rather than left multiplication is the appropriate Fischer partner; and Section~6.2 carries this same example through the complete homogeneous determinant calculation.}

The results are algebraic. The general Fischer theorem is formulated after central determinant localization. Closed determinant formulas are obtained in the one- and two-vector cases, while for arbitrary finite support the determinant is factored into exact-support determinants. \newrev{At a fixed numerical specialization \((q,M)\), the unlocalized decomposition follows whenever the corresponding determinant factors are nonzero. The explicit two-vector nonresonance result is therefore kept separate from the general finite-support statement.}

\section{Radial algebras, coefficient ring, and the coordinatewise obstruction}

\newrev{This section introduces the algebraic setting used throughout the paper. The coefficient \(Q\), the formal notation \(q^m\), and an integer dimension \(M\) are distinguished explicitly because localization and numerical specialization use them in different ways. We then show that independent coordinatewise Jackson dilations do not preserve the radial subalgebra, which motivates the whole-vector deformation introduced in Section~3.}

Let \(S\) be a set of abstract vector variables.  For algebra elements \(a,b\), write
\[
        \{a,b\}:=ab+ba,
        \qquad
        [a,b]:=ab-ba
\]
for the anticommutator and commutator, respectively.  The radial algebra \(R(S)\) is the associative algebra over \(\R\) generated by \(S\) subject to
\begin{equation}\label{eq:radial-relation}
        [\{u,v\},w]=0,
        \qquad u,v,w\in S.
\end{equation}
The scalar subalgebra \(R_0(S)\) is generated by the central anticommutators \(\{u,v\}\).  For every finite set \(\{x_1,\ldots,x_N\}\subset S\), each element has a unique expansion with coefficients in the scalar polynomial algebra and exterior factors \(x_{i_1}\wedge\cdots\wedge x_{i_p}\) \cite{Sommen1997,DeSchepperGuzmanAdanSommen2017}.

\subsection{Coefficient ring and dimension specializations}

\newrev{The deformation is formulated before choosing a numerical dimension. We keep \(q\) and the dimension parameter \(Q\) independent and later specialize \(Q=q^M\). This subsection also introduces the localization notation used to distinguish generic algebraic inverses from statements at a fixed numerical specialization.}

The deformation coefficients require a scalar extension.  Let \(q\) and \(Q\) be algebraically independent indeterminates and put
\begin{equation}\label{eq:coefficient-field}
        \B:=\R[q,q^{-1},Q,(1-q)^{-1}],
        \qquad
        \K:=\operatorname{Frac}(\B)=\R(q,Q),
        \qquad
        \RB(S):=\B\otimes_{\R}R(S).
\end{equation}
The symbol \(Q\) records the formal dimension.  For \(a\in\mathbb Z\), we use the notation
\begin{equation}\label{eq:formal-dimension-notation}
        q^{m+a}:=Qq^a,
        \qquad
        \qnum{m+a}:=\frac{1-Qq^a}{1-q},
\end{equation}
and for an integer \(n\),
\begin{equation}\label{eq:q-number}
        \qnum n:=\frac{1-q^n}{1-q}.
\end{equation}
Thus \(q^m\) is mnemonic notation for the independent coefficient \(Q\), not an algebraic relation in \(\B\) or its fraction field \(\K\).

For a positive integer \(M\), the numerical dimension specialization is the ring homomorphism
\begin{equation}\label{eq:dimension-specialization}
        \sigma_M:\B\longrightarrow\R(q),
        \qquad \sigma_M(Q)=q^M.
\end{equation}
After applying \(\sigma_M\), every expression \(\qnum{m+a}\) has a removable singularity at \(q=1\), with limit \(M+a\).  All classical-limit statements in this paper are understood in this sense: first specialize \(Q=q^M\), then let \(q\to1\).

\newrev{The following remark records the roles of \(Q\), \(m\), and \(M\) used throughout the paper.}

\newrev{\begin{remark}[The three dimension symbols]\label{rem:QmM}
The notation separates three roles:
\[
\begin{array}{c|c}
Q&\text{an independent central coefficient in }\B,\\
q^m&\text{mnemonic notation for }Q\text{ inside expressions such as }[m+a]_q,\\
M&\text{an actual positive integer used only after }Q=q^M.
\end{array}
\]
Thus statements over \(\B\) are formal in \(Q\); statements involving a Clifford dimension are made only after the specialization \(\sigma_M\). In particular, one should not manipulate \(m\) as an independent real variable inside \(\B\): only the combinations \(q^{m+a}=Qq^a\) and \([m+a]_q=(1-Qq^a)/(1-q)\) are defined algebraically.
\end{remark}}

All localizations below are taken with respect to multiplicative subsets of central coefficient rings.  Hence the usual commutative localization construction applies to the corresponding radial-algebra modules.  If \(\Omega\) is such a multiplicative set and \(\mathcal M\) is a module, we write
\[
        \mathcal M_{\Omega}:=\Omega^{-1}\mathcal M.
\]
For a single central element \(d\), the notation \(\mathcal M_d\) means localization at the multiplicative set \(\{1,d,d^2,\ldots\}\).

\newrev{Localization over the generic coefficient ring does not imply invertibility at every numerical specialization. A denominator \(d(q,Q)\) that is nonzero generically may satisfy \(d(q_0,q_0^M)=0\). For numerical specializations used below, the required nonvanishing conditions are therefore stated explicitly.}

For a scalar variable \(u\) and a parameter \(P\), write
\begin{equation}\label{eq:jackson}
        \delta^{(u)}_P f
        :=\frac{f(u)-f(Pu)}{(1-P)u},
        \qquad
        T^{(u)}_P f(u):=f(Pu).
\end{equation}
These are the standard Jackson difference and dilation operators; see, for example, \cite{KacCheung2002}.

\subsection{Failure of coordinatewise Jackson differentiation}

\newrev{Before defining an intrinsic radial operator, it is useful to see explicitly why the most direct coordinatewise construction fails. The calculation below uses only the scalar radial polynomial \(x^4\); hence the obstruction appears before any exterior or multi-vector complication. The one-vector whole-dilation model at the end of the subsection provides the prototype for the later relative construction.}

Take the Clifford convention \(e_ie_j+e_je_i=-2\delta_{ij}\), let \(x=\sum_{j=1}^M x_je_j\), and put \(\rho=\sum_jx_j^2\), so that \(x^2=-\rho\).  Define the coordinatewise right Jackson--Dirac operator
\begin{equation}\label{eq:coord-jackson}
        Q_x^{\mathrm R}F:=-\sum_{j=1}^M D_{q,x_j}F\,e_j,
        \qquad
        D_{q,x_j}F=\frac{F(x)-F(x_1,\ldots,qx_j,\ldots,x_M)}{(1-q)x_j}.
\end{equation}
With this right-handed sign convention,
\[
        Q_x^{\mathrm R}[x]=M,
        \qquad
        Q_x^{\mathrm R}[x^2]=\qnum2x.
\]
For comparison, the left-handed operator \(Q_x^{\mathrm L}F:=\sum_{j=1}^M e_jD_{q,x_j}F\) satisfies \(Q_x^{\mathrm L}[x]=-M\) and \(Q_x^{\mathrm L}[x^2]=-\qnum2x\).
For the scalar radial polynomial \(x^4=\rho^2\), one has
\[
D_{q,x_j}(\rho^2)
=(1+q)x_j\left((1+q^2)x_j^2+2\sum_{i\ne j}x_i^2\right).
\]
When \(M\ge2\) and \(q\ne1\), the coefficient of \(e_j\) in \(Q_x^{\mathrm R}(x^4)\) depends on \(x_j^2\) separately.  It therefore cannot be expressed as a scalar polynomial in \(x^2\) multiplied by \(x\).  Thus \(Q_x^{\mathrm R}\) does not preserve
\[
        \R[x^2]\oplus x\R[x^2].
\]
This proves that independent coordinate dilations do not define a radial \(q\)-vector derivative.

For one abstract vector, the unique exterior decomposition has the form
\[
        F=A(r)+xB(r),
        \qquad
        A,B\in\B[r],
        \qquad r=x^2.
\]
The whole-vector dilation \(x\mapsto qx\), equivalently \(r\mapsto q^2r\), motivates the requirement that powers be lowered by the corresponding \(q\)-numbers.  This leads to the following one-vector model.
\begin{equation}\label{eq:one-vector-closed}
\partial^{\mathrm{rad},\mathrm R}_{x,q}(A+xB)
=(1+q)x\delta^{(r)}_{q^2}A+\qnum m B+Q(1+q)r\delta^{(r)}_{q^2}B.
\end{equation}
Equivalently,
\[
\partial^{\mathrm{rad},\mathrm R}_{x,q}(x^{2n})=\qnum{2n}x^{2n-1}\quad(n\ge1),
\qquad
\partial^{\mathrm{rad},\mathrm R}_{x,q}(x^{2n+1})=\qnum{m+2n}x^{2n}\quad(n\ge0).
\]
The odd rule uses the identity \(\qnum{m+2n}=\qnum m+Q\qnum{2n}\), so it combines the initial value \(\partial^{\mathrm{rad},\mathrm R}_{x,q}(x)=\qnum m\) with the same whole-vector radial lowering as in even degree.  Thus \eqref{eq:one-vector-closed} is exactly the operator determined on the monomial basis by the two displayed rules.  After \(Q=q^M\) and \(q\to1\), it becomes the ordinary right radial vector derivative.

\section{The relative right \texorpdfstring{\(q\)}{q}-vector derivative}

\newrev{This section constructs the relative right \(q\)-vector derivative. We first derive the classical right derivative of an exterior blade. We then define the scalar Jackson operator on the central invariants, combine it with the exterior decomposition, establish the classical limit and low-degree formulas, and prove compatibility of the finite relative operators under enlargement of the parameter set.}

\newrev{\subsection{Right-action convention and the classical blade identity}}

\newrev{We first fix the left/right conventions independently of the deformation. In a Clifford realization of dimension \(M\) we use
\begin{equation}\label{eq:clifford-convention}
 e_ie_j+e_je_i=-2\delta_{ij},
 \qquad
 [F]\partial_x:=-\sum_{j=1}^M\frac{\partial F}{\partial x_j}e_j,
 \qquad x=\sum_{j=1}^Mx_je_j.
\end{equation}
Thus the basis vectors belonging to the differential operator occur on the \emph{right}. Scalar coefficients in the exterior normal form are central and are written to the right of the exterior factor. Later, the monogenic Fischer decomposition will use right multiplication \(R_xF=Fx\). By contrast, the auxiliary exterior-creation operator of Section~4 sends \(y_I\) to \(x\wedge y_I\); it is not full multiplication by \(x\).}

\newrev{For reference, the sides used in the paper are summarized by
\[
\begin{array}{c|c|c}
\text{object} & \text{notation} & \text{side/convention}\\ \hline
\text{classical derivative} &[F]\partial_x&\text{basis vector }e_j\text{ on the right}\\
\text{finite }q\text{-derivative}&\partial^{Y,\mathrm R}_{x,q}&\text{deforms }[F]\partial_x\\
\text{Fischer multiplication}&R_xF=Fx&\text{full multiplication on the right}\\
\text{exterior creation}&\Cop_x^Y(y_I A)=(x\wedge y_I)A&\text{auxiliary exterior operation}\\
\text{scalar coefficients}&A,B\in\A_x(Y)&\text{central; written on the right}
\end{array}
\]}

\newrev{The following lemma gives the classical right derivative of an exterior blade. Its coefficient determines the blade-lowering term in the \(q\)-deformation.}

\begin{lemma}[Classical right derivative of an exterior blade]\newrev{\label{lem:classical-blade}}
\newrev{Let \(A_p\) be an \(x\)-independent homogeneous exterior \(p\)-vector in a Clifford algebra satisfying \eqref{eq:clifford-convention}. Then
\begin{equation}\label{eq:classical-blade}
 [x\wedge A_p]\partial_x=(-1)^p(M-p)A_p.
\end{equation}}
\end{lemma}

\begin{proof}
\newrev{It is enough to verify the standard Clifford identity
\begin{equation}\label{eq:sandwich-identity}
 \sum_{j=1}^M e_jA_pe_j=(-1)^p(2p-M)A_p.
\end{equation}
For a basis blade \(A_p=e_{i_1}\cdots e_{i_p}\) with distinct indices, a term with \(j\notin\{i_1,\ldots,i_p\}\) anticommutes through all \(p\) factors and contributes \((-1)^{p+1}A_p\), whereas a term with \(j\in\{i_1,\ldots,i_p\}\) contributes \((-1)^pA_p\). There are \(M-p\) terms of the first type and \(p\) of the second, giving \eqref{eq:sandwich-identity}; linearity gives the general case.}

\newrev{Because \(A_p\) is independent of \(x\),
\[
 [xA_p]\partial_x=-\sum_j e_jA_pe_j=(-1)^p(M-2p)A_p,
 \qquad
 [A_px]\partial_x=-\sum_jA_pe_je_j=MA_p.
\]
For a vector and a homogeneous exterior \(p\)-vector,
\[
 x\wedge A_p=\frac12\bigl(xA_p+(-1)^pA_px\bigr).
\]
Applying the two preceding derivative identities yields
\[
 [x\wedge A_p]\partial_x
 =\frac12\left((-1)^p(M-2p)+(-1)^pM\right)A_p
 =(-1)^p(M-p)A_p.
\]
This proves \eqref{eq:classical-blade}.}
\end{proof}

\newrev{For \(p=1\), with \(w=x\wedge y=(xy-yx)/2\), formula~\eqref{eq:classical-blade} gives
\[
 [w]\partial_x=-(M-1)y.
\]
Equivalently, since \(\{x,y\}=xy+yx\), one has
\[
 [xy]\partial_x=(2-M)y,
 \qquad
 [yx]\partial_x=My.
\]
These identities are also obtained as the \(q\to1\) limits of the formulas in Example~\ref{ex:low-degree}.}

\subsection{Scalar Jackson calculus on the relative invariants}

\newrev{We next define the part of the derivative acting on central scalar coefficients. Only the invariants involving the distinguished vector \(x\) are dilated, while the internal coefficients \(c_{ij}\) remain fixed. The monomial formula~\eqref{eq:scalar-monomial-action} will be used repeatedly in the determinant computations of Section~6.}

Fix \(x\in S\) and a finite set \(Y=\{y_1,\ldots,y_N\}\subset S\setminus\{x\}\). Define
\begin{equation}\label{eq:relative-scalars}
 r=x^2,
 \qquad
 s_i=\{x,y_i\},
 \qquad
 c_{ij}=\{y_i,y_j\}=c_{ji},
\end{equation}
and let
\begin{equation}\label{eq:scalar-algebra}
 \A_x(Y):=\B[r,s_1,\ldots,s_N,c_{ij}:1\le i\le j\le N].
\end{equation}
The relative dilation is
\begin{equation}\label{eq:scalar-dilation}
 r\mapsto q^2r,
 \qquad
 s_i\mapsto qs_i,
 \qquad
 c_{ij}\mapsto c_{ij}.
\end{equation}
The scalar-input part of the right \(q\)-vector derivative is
\begin{equation}\label{eq:relative-vector}
 \dscal
 :=(1+q)x\delta^{(r)}_{q^2}
 +2T^{(r)}_{q^2}\sum_{i=1}^N y_i\delta^{(s_i)}_q.
\end{equation}
In particular,
\[
 \dscal(r)=\qnum2x,
 \qquad
 \dscal(s_i)=2y_i,
 \qquad
 \dscal(c_{ij})=0.
\]
More generally, for \(A=r^as^\beta P(c)\), where
\(P(c)\in\B[c_{ij}:1\le i\le j\le N]\),
\begin{equation}\label{eq:scalar-monomial-action}
\begin{aligned}
\dscal(A)
&=\qnum{2a}x r^{a-1}s^\beta P(c)\\
&\quad+2q^{2a}r^a\sum_{i:\,\beta_i>0}\qnum{\beta_i}\,y_i s^{\beta-e_i}P(c).
\end{aligned}
\end{equation}
Here the first term is omitted when \(a=0\), and \(e_i\) denotes the \(i\)-th standard multi-index. The first term lowers the radial degree according to \(r\mapsto q^2r\), while the shift \(T^{(r)}_{q^2}\) in the mixed terms supplies the factor \(q^{2a}\) dictated by the chosen ordered whole-vector dilation.

For \(I=\{i_1<\cdots<i_p\}\subset\{1,\ldots,N\}\), put
\[
 y_I:=y_{i_1}\wedge\cdots\wedge y_{i_p},
 \qquad y_\varnothing=1.
\]
The exterior decomposition relative to \(x\) is
\begin{equation}\label{eq:wedge-decomposition}
 \RB(\{x\}\cup Y)=
 \bigoplus_I y_I\A_x(Y)
 \oplus
 \bigoplus_I (x\wedge y_I)\A_x(Y).
\end{equation}

\newrev{\subsection{Definition, classical limit, and low-degree formulas}}

\newrev{Combining the scalar Jackson calculus with the exterior decomposition gives the finite relative operator. The blade-lowering term is determined by Lemma~\ref{lem:classical-blade}, while the second term differentiates the central scalar coefficient.}

\begin{definition}[Finite relative right \(q\)-vector derivative]\label{def:finite-derivative}
The right \(q\)-vector derivative relative to \(x\) and \(Y\) is the \(\B\)-linear operator
\[
 \dY:\RB(\{x\}\cup Y)\longrightarrow\RB(\{x\}\cup Y)
\]
defined, for \(p=|I|\), by
\begin{align}
\dY(y_IA)&=y_I\dscal(A),\label{eq:derivative-y}\\
{\color{black}\dY((x\wedge y_I)B)}&={\color{black}(-1)^p\qnum{m-p}y_IB
      +q^{m-p}(x\wedge y_I)\dscal(B),}\label{eq:derivative-x}
\end{align}
for \(A,B\in\A_x(Y)\).
\end{definition}

\newrev{By Lemma~\ref{lem:classical-blade}, the blade-lowering term in \eqref{eq:derivative-x} carries the factor \((-1)^p\). The scalar-differentiation term has no additional parity factor: in the classical limit it is \((x\wedge y_I)\) multiplied by the right derivative of the central scalar coefficient.}

For \(Y=\varnothing\), Definition~\ref{def:finite-derivative} is exactly the one-vector operator \eqref{eq:one-vector-closed}. It also gives
\begin{equation}\label{eq:pure-blade-actions}
 \dY(y_I)=0,
 \qquad
 \newrev{\dY(x\wedge y_I)=(-1)^{|I|}\qnum{m-|I|}y_I.}
\end{equation}

\newrev{The classical limit of the finite relative operator is given by the following proposition.}

\begin{proposition}[Classical limit]\label{prop:classical-limit-order}
Let \(M\) be a positive integer and specialize \(Q=q^M\). Then the limit \(q\to1\) of \(\dY\) exists and is the radial-algebra operator \(\dclass\) characterized on \eqref{eq:wedge-decomposition} by
\begin{align}
\dclass(y_IA)
 &=y_I\left(2x\frac{\partial A}{\partial r}
      +2\sum_{i=1}^Ny_i\frac{\partial A}{\partial s_i}\right),\label{eq:classical-right-y}\\
{\color{black}\dclass((x\wedge y_I)B)}
 &={\color{black}(-1)^{|I|}(M-|I|)y_IB+(x\wedge y_I)
 \left(2x\frac{\partial B}{\partial r}
      +2\sum_{i=1}^Ny_i\frac{\partial B}{\partial s_i}\right).}\label{eq:classical-right-x}
\end{align}
\newrev{Under every \(M\)-dimensional Clifford realization of the radial algebra, this operator agrees with the standard right Clifford derivative \([F]\partial_x=-\sum_j(\partial F/\partial x_j)e_j\).}
\end{proposition}

\begin{proof}
The Jackson differences in \eqref{eq:relative-vector} converge to the corresponding ordinary partial derivatives after \(Q=q^M\), and \(q^{m-|I|}\to1\). \newrev{Moreover,
\[
 (-1)^{|I|}\qnum{m-|I|}\longrightarrow (-1)^{|I|}(M-|I|).
\]
Thus the limit on a pure exterior blade agrees exactly with Lemma~\ref{lem:classical-blade}. For a general scalar coefficient \(B\), centrality of \(B\) gives the ordinary right Leibniz rule
\[
 [(x\wedge y_I)B]\partial_x
 =[x\wedge y_I]\partial_x\,B
 +(x\wedge y_I)[B]\partial_x,
\]
and the scalar right derivative is precisely the limit of \(\dscal\).} This proves \eqref{eq:classical-right-y}--\eqref{eq:classical-right-x}.
\end{proof}

\newrev{The coordinate comparison in Proposition~\ref{prop:classical-limit-order} can be seen directly. If
\[
 y_i=\sum_{j=1}^M y_{ij}e_j,
\]
then in the Clifford realization
\[
 r=-\sum_{j=1}^M x_j^2,
 \qquad
 s_i=-2\sum_{j=1}^M x_jy_{ij},
 \qquad
 c_{ij}=-2\sum_{\ell=1}^M y_{i\ell}y_{j\ell}.
\]
Hence the ordinary chain rule for a central scalar coefficient gives
\[
 [A(r,s,c)]\partial_x
 =2x\frac{\partial A}{\partial r}
 +2\sum_{i=1}^Ny_i\frac{\partial A}{\partial s_i},
\]
which is exactly \eqref{eq:classical-right-y}; combining this with Lemma~\ref{lem:classical-blade} gives \eqref{eq:classical-right-x}. Thus the abstract classical-limit formulas agree term by term with their coordinate realizations, including the right-action signs.}

\newrev{The following example gives the action on the lowest-degree elements and compares the abstract formulas with their classical limits.}

\begin{example}[Low-degree formulas]\newrev{\label{ex:low-degree}}
\newrev{Take one auxiliary vector \(Y=\{y\}\), and write
\[
 s=\{x,y\},\qquad w=x\wedge y=\frac{xy-yx}{2}.
\]
We compute each entry directly from Definition~\ref{def:finite-derivative}. First,
\[
 \partial^{\{y\},\mathrm R}_{x,q}(1)=0,
 \qquad
 \partial^{\{y\},\mathrm R}_{x,q}(y)=0,
\]
because their scalar coefficients are constant. Since \(x=x\wedge1\), the case \(p=0\) of \eqref{eq:derivative-x} gives
\[
 \partial^{\{y\},\mathrm R}_{x,q}(x)=\qnum m.
\]
The mixed scalar satisfies \(\dscal(s)=2y\), and therefore
\[
 \partial^{\{y\},\mathrm R}_{x,q}(s)=2y.
\]
For the one-blade \(w=x\wedge y\), we have \(p=1\) and constant scalar coefficient, so
\[
 \partial^{\{y\},\mathrm R}_{x,q}(w)=-\qnum{m-1}y.
\]
Finally, use the exterior/scalar decompositions
\[
 xy=w+\frac{s}{2},\qquad yx=-w+\frac{s}{2}.
\]
Linearity gives
\begin{align*}
 \partial^{\{y\},\mathrm R}_{x,q}(xy)
 &= -\qnum{m-1}y+\frac12(2y)
   =\bigl(1-\qnum{m-1}\bigr)y,\\
 \partial^{\{y\},\mathrm R}_{x,q}(yx)
 &= \qnum{m-1}y+\frac12(2y)
   =\bigl(1+\qnum{m-1}\bigr)y.
\end{align*}
Thus all seven elementary values are
\[
\begin{array}{c|c|c}
F&\partial^{\{y\},\mathrm R}_{x,q}F&\displaystyle\lim_{q\to1,\,Q=q^M}\partial^{\{y\},\mathrm R}_{x,q}F\\ \hline
1&0&0\\
x&\qnum m&M\\
y&0&0\\
\{x,y\}=s&2y&2y\\
x\wedge y=w&-\qnum{m-1}y&-(M-1)y\\
xy&\bigl(1-\qnum{m-1}\bigr)y&(2-M)y\\
yx&\bigl(1+\qnum{m-1}\bigr)y&My
\end{array}
\]
The relations are mutually consistent: adding the last two rows recovers the row for \(s=xy+yx\), while subtracting them and dividing by two recovers the row for \(w=(xy-yx)/2\).}

\newrev{The parity change can also be seen one degree higher. For two independent auxiliary vectors \(y_1,y_2\), put \(A_2=y_1\wedge y_2\). Then \eqref{eq:pure-blade-actions} gives
\[
 \partial^{\{y_1,y_2\},\mathrm R}_{x,q}(x\wedge A_2)
 =+\qnum{m-2}A_2,
\]
whose classical limit is \((M-2)A_2\). Hence the sign alternates exactly with exterior degree, as Lemma~\ref{lem:classical-blade} requires.}
\end{example}

\newrev{In particular, after \(Q=q^M\) and \(q\to1\), the formula for \(yx\) gives \([yx]\partial_x=My\). The element \(yx=R_x(y)\) reappears in the right Fischer operator of Section~5.}

\newrev{\subsection{Relabelling covariance and the direct limit}}

\newrev{To obtain an intrinsic operator by direct limit, the finite construction must be invariant under relabelling and compatible with enlargement of the parameter set. These properties are established next.}

\begin{proposition}[Relabelling covariance]\label{prop:relabelling}
Every bijection \(\tau:Y\to Y'\) induces a radial-algebra isomorphism that intertwines \(\partial^{Y,\mathrm R}_{x,q}\) and \(\partial^{Y',\mathrm R}_{x,q}\). In particular, the operator is independent of the ordering used to write the exterior basis.
\end{proposition}

\begin{proof}
A bijection \(\tau\) permutes the variables \(s_i\), the coefficients \(c_{ij}\), and the exterior generators. Every term in \eqref{eq:relative-vector}, \eqref{eq:derivative-y}, and \eqref{eq:derivative-x} is equivariant under this simultaneous relabelling. \newrev{The parity factor \((-1)^{|I|}\) depends only on the exterior degree and is therefore invariant under relabelling.} Changing the chosen order of \(Y\) only changes the signs used to represent the same exterior elements, and both sides of the defining formulas transform with the same signs.
\end{proof}

\newrev{We next establish compatibility under inclusions of finite parameter sets.}

\begin{theorem}[Direct-limit right \(q\)-vector derivative]\label{thm:direct-limit-derivative}
For fixed \(x\in S\), the finite operators \(\partial^{Y,\mathrm R}_{x,q}\) are compatible under inclusions of finite sets
\[
 Y\subset Z\subset S\setminus\{x\}.
\]
Consequently, they define a \(\B\)-linear operator
\begin{equation}\label{eq:direct-limit-derivative}
 \dglob:\RB(S)\longrightarrow\RB(S)
\end{equation}
for every radial algebra.
\end{theorem}

\begin{proof}
Every element of \(\RB(S)\) involves only finitely many vector variables. Choose \(Y\) containing all variables different from \(x\) that occur in the element. If \(Y\) is enlarged to \(Z\), its exterior expansion has no components containing the new variables, and every Jackson difference with respect to a new mixed scalar vanishes. Hence \eqref{eq:derivative-y}--\eqref{eq:derivative-x} give the same value in the smaller and larger finite radial algebras. Proposition~\ref{prop:relabelling} shows that the resulting value is independent of all auxiliary choices.
\end{proof}

\newrev{The role of the classical limit in the later arguments is summarized in the following remark.}

\newrev{\begin{remark}[Dependence on the classical limit]\label{rem:dependency}
The construction of \(\dY\), relabelling covariance, and the direct-limit theorem are algebraic statements over \(\B\). Proposition~\ref{prop:classical-limit-order} is used to identify their meaning as a deformation of the standard right derivative. The triangular calculations of Section~4 use only Definition~\ref{def:finite-derivative}; they do not use a Clifford realization. The classical limit re-enters in Proposition~\ref{prop:det-nonzero}, where it supplies one nonsingular specialization of the right-multiplication Fischer operator.
\end{remark}}

\section{Exterior creation and a localized Green decomposition}

\newrev{This section studies the anticommutator of the derivative with exterior creation. Its scalar part is diagonal and the remaining term is triangular with respect to the two exterior sectors, which gives a Green inverse after localization. This construction is distinct from the right-multiplication Fischer decomposition considered in Section~5.}

For finite \(Y\), write
\[
 \Ycal_{x;Y}:=\bigoplus_I y_I\A_x(Y),
 \qquad
 \Xcal_{x;Y}:=\bigoplus_I (x\wedge y_I)\A_x(Y).
\]
Define exterior creation by
\begin{equation}\label{eq:creation}
 \Cop_x^Y(y_IA)=(x\wedge y_I)A,
 \qquad
 \Cop_x^Y((x\wedge y_I)B)=0.
\end{equation}
The operators \(\Cop_x^Y\) are compatible under enlargement of \(Y\), and hence induce a direct-limit operator \(\Cop_x\) on \(\RB(S)\). Set
\begin{equation}\label{eq:H}
 \Hcal^Y_{x,q}:=\dY\Cop_x^Y+\Cop_x^Y\dY.
\end{equation}

For every scalar variable \(u\), let \(M_u\) denote multiplication by \(u\). Define the scalar \(q\)-Euler operator
\begin{equation}\label{eq:q-euler}
 E^{\mathrm{sc},Y}_{x,q}
 =(1+q)M_r\delta^{(r)}_{q^2}
 +T^{(r)}_{q^2}\sum_{i=1}^NM_{s_i}\delta^{(s_i)}_q.
\end{equation}
\newrev{For exterior degree \(p\ge0\), put
\begin{equation}\label{eq:Lambda-p}
 \Lambda^Y_{p,q}
 :=\qnum{m-p}I+q^{m-p}E^{\mathrm{sc},Y}_{x,q},
\end{equation}
and
\begin{equation}\label{eq:Delta-p}
 \Delta^Y_{p,q}:=(-1)^p\Lambda^Y_{p,q}.
\end{equation}
The overall parity sign in \(\Delta^Y_{p,q}\) comes from the blade-lowering term in \eqref{eq:derivative-x}.}

On \(r^as^\beta P(c)\), where
\[
 s^\beta=s_1^{\beta_1}\cdots s_N^{\beta_N},
 \qquad
 P(c)\in\B[c_{ij}:1\le i\le j\le N],
\]
the scalar operator \(\Lambda^Y_{p,q}\) has eigenvalue
\begin{equation}\label{eq:lambda}
 \newrev{\lambda_{p,a,\beta}(q,Q)
 =\qnum{m-p}+q^{m-p}
 \left(\qnum{2a}+q^{2a}\sum_{i=1}^N\qnum{\beta_i}\right),}
\end{equation}
and hence \(\Delta^Y_{p,q}\) has eigenvalue
\begin{equation}\label{eq:mu}
 \newrev{\mu_{p,a,\beta}(q,Q)=(-1)^p\lambda_{p,a,\beta}(q,Q).}
\end{equation}

\newrev{To identify the diagonal block of the anticommutator, we need only the component of \((x\wedge y_I)\dscal(A)\) that loses its exterior \(x\). The following projection lemma isolates exactly that contribution.}

\begin{lemma}[Projection of the differentiated scalar term]\label{lem:no-x-projection}
Let \(I\subset\{1,\ldots,N\}\), \(p=|I|\), and \(A\in\A_x(Y)\). Then
\begin{equation}\label{eq:no-x-projection}
\operatorname{pr}_{\Ycal_{x;Y}}
\left((x\wedge y_I)\dscal(A)\right)
=(-1)^py_I E^{\mathrm{sc},Y}_{x,q}A.
\end{equation}
\end{lemma}

\begin{proof}
Write
\[
\dscal(A)=x\alpha+\sum_{j=1}^Ny_j\beta_j,
\qquad
\alpha=(1+q)\delta^{(r)}_{q^2}A,
\qquad
\beta_j=2T^{(r)}_{q^2}\delta^{(s_j)}_qA.
\]
\newrev{We only need the component containing no exterior \(x\). Multiplication on the right by a vector can either increase exterior degree or contract one of the vectors already present. In \((x\wedge y_I)x\), the no-\(x\) component arises from contracting the final \(x\) with the distinguished \(x\); moving the final vector past the \(p\) vectors in \(y_I\) contributes \((-1)^p\), and the contraction contributes the central scalar \(r=x^2\). Thus
\[
 \operatorname{pr}_{\Ycal}\bigl((x\wedge y_I)x\bigr)=(-1)^pry_I.
\]
Similarly, in \((x\wedge y_I)y_j\), the only contribution to the same exterior blade \(y_I\) comes from contracting \(y_j\) with the distinguished \(x\), which yields \(s_j/2\); the same passage through \(p\) exterior factors gives
\[
 \operatorname{pr}_{\Ycal}\bigl((x\wedge y_I)y_j\bigr)
 =(-1)^p\frac{s_j}{2}y_I.
\]
Contractions with a vector in \(y_I\) lower the \(y\)-support and leave an exterior \(x\), while exterior products also remain in \(\Xcal_{x;Y}\); neither affects the displayed projection.}

Therefore
\[
\begin{aligned}
\operatorname{pr}_{\Ycal}
\left((x\wedge y_I)\dscal(A)\right)
&=(-1)^py_I\left(r\alpha+\frac12\sum_js_j\beta_j\right)\\
&=(-1)^py_IE^{\mathrm{sc},Y}_{x,q}A.
\end{aligned}
\]
\end{proof}

\newrev{By the projection lemma, the same scalar diagonal acts on both exterior sectors, while the remaining terms map the no-\(x\) sector into the \(x\)-sector. This gives the following block-triangular form.}

\begin{theorem}[Triangular exterior anticommutator]\label{thm:triangular-H}
Let \(I\subset\{1,\ldots,N\}\), \(p=|I|\), and \(A\in\A_x(Y)\). Then
\begin{align}
 \Hcal^Y_{x,q}((x\wedge y_I)A)
 &=(x\wedge y_I)\Delta^Y_{p,q}A,
 \label{eq:H-x}\\
 \Hcal^Y_{x,q}(y_IA)
 &=y_I\Delta^Y_{p,q}A+\mathcal N^Y_{I,q}(A),
 \label{eq:H-y}
\end{align}
where
\begin{equation}\label{eq:N-explicit}
\begin{aligned}
\mathcal N^Y_{I,q}(A)
&:=q^{m-p}(x\wedge y_I)\dscal(A)
 +\Cop_x^Y\bigl(y_I\dscal(A)\bigr)\\
&\quad-(-1)^pq^{m-p}y_IE^{\mathrm{sc},Y}_{x,q}A
\in\Xcal_{x;Y}.
\end{aligned}
\end{equation}
Let \(\mathcal N^Y_{x,q}:\Ycal_{x;Y}\to\Xcal_{x;Y}\) be determined by \(\mathcal N^Y_{x,q}(y_IA)=\mathcal N^Y_{I,q}(A)\). Hence, relative to \(\Ycal_{x;Y}\oplus\Xcal_{x;Y}\),
\begin{equation}\label{eq:H-block}
 \Hcal^Y_{x,q}=
 \begin{pmatrix}
 \bigoplus_I\Delta^Y_{|I|,q}&0\\
 \mathcal N^Y_{x,q}&\bigoplus_I\Delta^Y_{|I|,q}
 \end{pmatrix}.
\end{equation}
\end{theorem}

\begin{proof}
For an \(x\)-sector element, \(\Cop_x^Y\) vanishes and
\[
\Hcal^Y_{x,q}((x\wedge y_I)A)
=\Cop_x^Y\dY((x\wedge y_I)A).
\]
By formula \eqref{eq:derivative-x},
\[
\newrev{\dY((x\wedge y_I)A)
=(-1)^p\qnum{m-p}y_IA+q^{m-p}(x\wedge y_I)\dscal(A).}
\]
Exterior creation sends the first term to
\((-1)^p\qnum{m-p}(x\wedge y_I)A\). For the second term, \(\Cop_x^Y\) first discards the \(x\)-sector component and then adjoins exterior \(x\). Lemma~\ref{lem:no-x-projection} therefore contributes
\[
(-1)^pq^{m-p}(x\wedge y_I)E^{\mathrm{sc},Y}_{x,q}A.
\]
Adding the two pieces gives \((x\wedge y_I)\Delta^Y_{p,q}A\), proving \eqref{eq:H-x}.

For \(y_IA\),
\[
\begin{aligned}
\Hcal^Y_{x,q}(y_IA)
&=\dY((x\wedge y_I)A)+\Cop_x^Y(y_I\dscal(A))\\
&=\newrev{(-1)^p\qnum{m-p}y_IA}
 +q^{m-p}(x\wedge y_I)\dscal(A)
 +\Cop_x^Y(y_I\dscal(A)).
\end{aligned}
\]
Lemma~\ref{lem:no-x-projection} identifies the \(\Ycal\)-component of the middle term as \((-1)^pq^{m-p}y_IE^{\mathrm{sc},Y}_{x,q}A\). Hence the full \(\Ycal\)-component is exactly \(y_I\Delta^Y_{p,q}A\), and the remainder is \eqref{eq:N-explicit} in \(\Xcal_{x;Y}\). This proves \eqref{eq:H-y} and the block form.
\end{proof}

\newrev{Before constructing the inverse, consider the case of one auxiliary vector. It displays the diagonal factors in the first nontrivial exterior degrees.}

\newrev{\begin{example}[One auxiliary vector]\label{ex:H-one-y}
Let \(Y=\{y\}\), \(w=x\wedge y\), and \(A=r^as^bP(t)\). On exterior degree \(p=1\),
\[
 \Delta^Y_{1,q}A
 =-\left(\qnum{m-1}+q^{m-1}\bigl(\qnum{2a}+q^{2a}\qnum b\bigr)\right)A.
\]
In particular,
\[
 \Hcal^{\{y\}}_{x,q}(y)=-\qnum{m-1}y,
 \qquad
 \Hcal^{\{y\}}_{x,q}(w)=-\qnum{m-1}w.
\]
After \(Q=q^M\) and \(q\to1\), both diagonal coefficients tend to \(-(M-1)\), consistently with the right-action sign in Example~\ref{ex:low-degree}.
\end{example}}

Let \(\Omega_{x;Y}\subset\B\) be the multiplicative set generated by all nonzero formal factors \(\lambda_{p,a,\beta}(q,Q)\) from \eqref{eq:lambda}. \newrev{Since \(\mu_{p,a,\beta}=(-1)^p\lambda_{p,a,\beta}\), localizing at \(\lambda\) or at \(\mu\) is equivalent.} Define the graded scalar Green operator by
\begin{equation}\label{eq:graded-green}
G^{\mathrm{gr},Y}_{x,q}
\bigl(y_Ir^as^\beta P(c)\bigr)
=\newrev{\mu_{|I|,a,\beta}(q,Q)^{-1}}y_Ir^as^\beta P(c),
\end{equation}
with the same formula on \((x\wedge y_I)r^as^\beta P(c)\).

\newrev{The Green inverse is constructed over the generic coefficient ring, so one must distinguish invertibility there from invertibility after fixing \(M\) and \(q\). The following remark makes that distinction explicit before the inverse is stated.}

\newrev{\begin{remark}[Formal localization versus numerical specialization]\label{rem:formal-specialization}
The localization above is a statement over the generic coefficient ring with independent \(q\) and \(Q\). After a numerical dimension specialization \(Q=q^M\), a formally nonzero denominator may specialize to zero. Therefore a specialized Green operator is defined only on blocks for which the corresponding numerical factors are nonzero. For \(0<q<1\) and \(M>p\), every term in
\[
 \lambda_{p,a,\beta}(q,q^M)=\qnum{M-p}+q^{M-p}
 \left(\qnum{2a}+q^{2a}\sum_i\qnum{\beta_i}\right)
\]
is nonnegative and the first is strictly positive, so the factor cannot vanish. If \(M=p\), the factor vanishes precisely in scalar degree \(a=0\), \(\beta=0\). For \(p>M\), cancellation can occur. This distinction is why the formal localized theorem and its specialized versions are stated separately.
\end{remark}}

\newrev{After adjoining inverses of the displayed diagonal factors, the triangular block matrix can be inverted by an elementary ordered block calculation. This yields the finite Green operator and, by compatibility, its direct-limit version.}

\begin{theorem}[Localized exterior Green operator]\label{thm:green}
The operator \(\Hcal^Y_{x,q}\) is invertible on
\(\RB(\{x\}\cup Y)_{\Omega_{x;Y}}\). Its inverse is
\begin{equation}\label{eq:green-matrix}
 \Gcal^Y_{x,q}=
 \begin{pmatrix}
 G^{\mathrm{gr},Y}_{x,q}&0\\
 -G^{\mathrm{gr},Y}_{x,q}\mathcal N^Y_{x,q}G^{\mathrm{gr},Y}_{x,q}
 &G^{\mathrm{gr},Y}_{x,q}
 \end{pmatrix}.
\end{equation}
The finite inverses are compatible under enlargement of \(Y\). Consequently, if \(\Omega_x\) is generated by all finite diagonal factors, they induce a direct-limit inverse \(\Gcal_{x,q}\) of
\[
 \Hcal_{x,q}:=\dglob\Cop_x+\Cop_x\dglob
\]
on \(\RB(S)_{\Omega_x}\).
\end{theorem}

\begin{proof}
Put \(\Delta=\bigoplus_I\Delta^Y_{|I|,q}\), \(N=\mathcal N^Y_{x,q}\), and \(G=G^{\mathrm{gr},Y}_{x,q}=\Delta^{-1}\). Then
\[
\begin{pmatrix}\Delta&0\\N&\Delta\end{pmatrix}
\begin{pmatrix}G&0\\-GNG&G\end{pmatrix}
=\begin{pmatrix}I&0\\NG-\Delta GNG&I\end{pmatrix}=I,
\]
and in the reverse order,
\[
\begin{pmatrix}G&0\\-GNG&G\end{pmatrix}
\begin{pmatrix}\Delta&0\\N&\Delta\end{pmatrix}
=\begin{pmatrix}I&0\\-GNG\Delta+GN&I\end{pmatrix}=I.
\]
\newrev{No commutation between \(N\) and \(G\) is being assumed: the order in \(-GNG\) is exactly the one required by block-matrix inversion.} This verifies the finite inverse explicitly. If \(Y\subset Z\), the subalgebra generated by \(x\) and \(Y\) is invariant under both finite anticommutators, their restrictions agree, and both diagonal inverses act by the same factors on monomials supported in \(Y\). Therefore the Green inverses agree on the smaller subalgebra after localization by the union of the finite denominator sets.
\end{proof}

\newrev{Expanding the identity \(\Hcal_{x,q}\Gcal_{x,q}=I\) separates every element into its two Green components and produces complementary projections.}

\begin{corollary}[Exterior Green decomposition]\label{cor:exterior-fischer}
On \(\RB(S)_{\Omega_x}\),
\begin{equation}\label{eq:green-identity}
 F=\dglob(\Cop_x\Gcal_{x,q}F)+\Cop_x(\dglob\Gcal_{x,q}F).
\end{equation}
The operators
\begin{equation}\label{eq:green-projectors}
 P_{\partial}:=\dglob\Cop_x\Gcal_{x,q},
 \qquad
 P_{\Cop}:=\Cop_x\dglob\Gcal_{x,q}
\end{equation}
are complementary projections. Hence
\begin{equation}\label{eq:exterior-decomp}
 \RB(S)_{\Omega_x}=\im P_{\partial}\oplus\im P_{\Cop},
 \qquad
 \im P_{\partial}\subseteq\im\dglob,
 \quad
 \im P_{\Cop}\subseteq\im\Cop_x.
\end{equation}
\end{corollary}

\begin{proof}
Equation \eqref{eq:green-identity} is the identity
\(I=\Hcal_{x,q}\Gcal_{x,q}\) expanded using \eqref{eq:H}. On a finite subalgebra, write
\[
 \dY=\begin{pmatrix}A&\Delta\\ C&D\end{pmatrix},
 \qquad
 \Cop_x^Y=\begin{pmatrix}0&0\\ I&0\end{pmatrix}
\]
relative to \(\Ycal_{x;Y}\oplus\Xcal_{x;Y}\), where
\[
\begin{aligned}
A&=\operatorname{pr}_{\Ycal}\dY\bigm|_{\Ycal},&
\Delta&=\operatorname{pr}_{\Ycal}\dY\bigm|_{\Xcal}
 =\bigoplus_I\Delta^Y_{|I|,q},\\
C&=\operatorname{pr}_{\Xcal}\dY\bigm|_{\Ycal},&
D&=\operatorname{pr}_{\Xcal}\dY\bigm|_{\Xcal}.
\end{aligned}
\]
Since \(\Hcal^Y_{x,q}=\dY\Cop_x^Y+\Cop_x^Y\dY\), its lower-left block is \(\mathcal N^Y_{x,q}=A+D\). Substitution of \eqref{eq:green-matrix} gives
\[
 P_{\partial}=
 \begin{pmatrix}I&0\\D\Delta^{-1}&0\end{pmatrix},
 \qquad
 P_{\Cop}=
 \begin{pmatrix}0&0\\-D\Delta^{-1}&I\end{pmatrix}.
\]
These matrices are idempotent, have zero product in both orders, and sum to the identity. Compatibility under finite enlargement gives the global statement.
\end{proof}

\section{Determinant-localized right \texorpdfstring{\(q\)}{q}-monogenic Fischer decomposition}

\newrev{This section develops the right \(q\)-monogenic Fischer decomposition. We first write right multiplication by \(x\) in the exterior normal form, then prove generic invertibility of \(\partial^{Y,\mathrm R}_{x,q}R_x\), and finally obtain finite and global decompositions after determinant localization.}

\newrev{The exterior creation operator \(\Cop_x^Y\) of Section~4 is not full multiplication by \(x\). For the right Fischer decomposition we use full right multiplication and set}
\begin{equation}\label{eq:Rx}
 \newrev{R_x:F\longmapsto Fx.}
\end{equation}
For \(I=\{i_1<\cdots<i_p\}\), define contraction by
\begin{equation}\label{eq:iota}
 \iota_x(y_I)=\sum_{\nu=1}^{p}(-1)^{\nu-1}\frac{s_{i_\nu}}2
 y_{I\setminus\{i_\nu\}}.
\end{equation}

\newrev{The determinant computations of Section~6 use the exterior normal basis, so we first express right multiplication by \(x\) in that basis.}

\begin{lemma}[Right multiplication in exterior normal form]\newrev{\label{lem:Rx-formulas}}
\newrev{For \(p=|I|\) and \(A\in\A_x(Y)\),
\begin{align}
 R_x(y_IA)&=(-1)^p\bigl((x\wedge y_I)A-\iota_x(y_I)A\bigr),\label{eq:Rx-y}\\
 R_x((x\wedge y_I)A)&=(-1)^p\bigl(r y_IA-(x\wedge\iota_x(y_I))A\bigr).
 \label{eq:Rx-x}
\end{align}}
\end{lemma}

\begin{proof}
\newrev{The scalar coefficient \(A\) is central, so it is enough to work with the exterior factors. The familiar vector--blade decomposition in the radial algebra is
\[
 x y_I=x\wedge y_I+\iota_x(y_I).
\]
The corresponding blade--vector formula is obtained by moving the final \(x\) through a homogeneous exterior \(p\)-blade. The exterior part acquires \((-1)^p\), while each contraction term acquires the opposite relative sign; hence
\[
 y_Ix=(-1)^p\bigl(x\wedge y_I-\iota_x(y_I)\bigr),
\]
which is \eqref{eq:Rx-y}. These formulas follow directly from the central relations \(y_jx=s_j-xy_j\) and antisymmetrization. For example,
\[
 yx=-x\wedge y+\frac{s}{2},
\]
and, for \(I=\{1,2\}\),
\[
 (y_1\wedge y_2)x
 =x\wedge y_1\wedge y_2
 -\frac{s_1}{2}y_2+\frac{s_2}{2}y_1,
\]
which displays the same pattern.}

\newrev{For the second formula, apply the same blade--vector decomposition to the \((p+1)\)-blade \(x\wedge y_I\). Its exterior product with the final \(x\) vanishes because \(x\wedge x=0\). The contraction with the distinguished \(x\) gives \((-1)^p r y_I\), while contractions with the vectors in \(y_I\) give \((-1)^{p+1}x\wedge\iota_x(y_I)\). Thus
\[
 (x\wedge y_I)x=(-1)^p\bigl(r y_I-x\wedge\iota_x(y_I)\bigr),
\]
which is \eqref{eq:Rx-x}. For one auxiliary vector, with \(w=x\wedge y\), this gives
\[
 yx=-w+\frac{s}{2},
 \qquad
 wx=\frac{s}{2}x-r y,
\]
in agreement with the two formulas above.}
\end{proof}

Give \(\RB(\{x\}\cup Y)\) the \(x\)-degree
\begin{equation}\label{eq:x-degree}
 \deg_x r=2,
 \quad \deg_x s_i=1,
 \quad \deg_x c_{ij}=0,
 \quad \deg_x y_I=0,
 \quad \deg_x(x\wedge y_I)=1.
\end{equation}
Let \(H^Y_n\) be the homogeneous component of degree \(n\), and put
\begin{equation}\label{eq:TY}
 \T_Y:=\B[c_{ij}:1\le i\le j\le N].
\end{equation}
Then \(H^Y_n\) is a finite free \(\T_Y\)-module with basis
\begin{equation}\label{eq:H-basis}
 y_I r^as^\beta\quad (2a+|\beta|=n),
 \qquad
 (x\wedge y_I)r^as^\beta\quad (1+2a+|\beta|=n).
\end{equation}
The formulas above imply
\[
 \newrev{R_xH^Y_n\subseteq H^Y_{n+1}},
 \qquad
 \dY H^Y_{n+1}\subseteq H^Y_n.
\]
An element \(F\in H^Y_n\) is called \emph{right \(q\)-monogenic relative to \(x\) and \(Y\)} if \(\dY F=0\). Likewise, an element of \(\RB(S)\) is right \(q\)-monogenic if it belongs to the kernel of the direct-limit operator \(\dglob\). Since \(\dY\) and \(R_x\) are \(\T_Y\)-linear, they extend uniquely to every central localization of \(\T_Y\).

Define
\begin{equation}\label{eq:K-D}
 \newrev{\Kop^Y_k:=\dY R_x\bigm|_{H^Y_k}:H^Y_k\to H^Y_k,}
 \qquad
 D^Y_k:=\det\Kop^Y_k\in\T_Y.
\end{equation}

\newrev{The Fischer decomposition requires injectivity of right multiplication. We record this property before considering the homogeneous determinants.}

\begin{lemma}[Injectivity of right multiplication]\newrev{\label{lem:Rx-injective}}
\newrev{The map \(R_x\) is injective on every finite radial algebra over \(\B\), and remains injective after central localization. It also remains injective after the numerical specializations \(Q=q^M\), \(q=q_0\in(0,1)\) used below.}
\end{lemma}

\begin{proof}
\newrev{If \(Fx=0\), multiply once more on the right by \(x\). Associativity and centrality of \(r=x^2\) give
\[
 0=(Fx)x=F x^2=Fr=rF.
\]
The exterior decomposition \eqref{eq:wedge-decomposition} makes the finite radial algebra a free module over its scalar polynomial algebra, in which the indeterminate \(r\) is not a zero divisor. Therefore \(F=0\). If the coefficient ring is centrally localized, the same exterior basis remains free and multiplication by \(r\) remains injective, so the argument is unchanged. After \(Q=q^M\) and \(q=q_0\in(0,1)\), the same normal form is a free module over the resulting scalar polynomial algebra, and the identical argument applies.}
\end{proof}

\newrev{The distinction between right and left multiplication is already visible in degree zero. The following example compares the two products.}

\begin{example}[Right and left multiplication in degree zero]\newrev{\label{ex:right-pairing}}
\newrev{Continue Example~\ref{ex:low-degree} with \(Y=\{y\}\). Right multiplication gives
\[
 R_x(y)=yx=-w+\frac{s}{2}.
\]
Applying the derivative term by term,
\[
 \partial^{\{y\},\mathrm R}_{x,q}(-w)=\qnum{m-1}y,
 \qquad
 \partial^{\{y\},\mathrm R}_{x,q}\left(\frac{s}{2}\right)=y,
\]
and hence
\[
 \Kop^{\{y\}}_0(y)
 =\partial^{\{y\},\mathrm R}_{x,q}(yx)
 =\bigl(1+\qnum{m-1}\bigr)y.
\]
After \(Q=q^M\) and \(q\to1\), this becomes \(My\), which is nonzero in every positive dimension; at \(M=2\) it equals \(2y\). By contrast,
\[
 xy=w+\frac{s}{2}
 \quad\Longrightarrow\quad
 [xy]\partial_x=-(M-1)y+y=(2-M)y,
\]
which vanishes at \(M=2\). Thus right and left multiplication give different degree-zero Fischer operators.}
\end{example}

\newrev{We now consider the homogeneous operator \(\Kop^Y_k=\partial^{Y,\mathrm R}_{x,q}R_x\). To obtain the localized Fischer decomposition, it is enough to show that its determinant is not identically zero; the classical right Fischer decomposition provides a nonsingular specialization.}

\begin{proposition}[Generic nonvanishing]\label{prop:det-nonzero}
For every finite \(Y\) and every \(k\ge0\), the determinant \(D^Y_k\) is not the zero element of \(\T_Y\).
\end{proposition}

\begin{proof}
Assume, to the contrary, that \(D^Y_k=0\) in \(\T_Y\). Then every specialization of \(D^Y_k\), whenever the defining matrix is specialized coefficientwise, is zero. We construct one specialization for which the classical limit of the matrix is nonsingular.

\newrev{Write \(Y=\{y_1,\ldots,y_N\}\) and take \(M=N+1\). Let \(e_1,\ldots,e_M\) be an orthonormal Clifford basis with
\[
 e_ie_j+e_je_i=-2\delta_{ij}.
\]
Specialize the internal scalar coefficients by realizing
\[
 y_i=e_i\quad(1\le i\le N),
 \qquad
 x=\sum_{j=1}^M x_je_j.
\]
Then
\begin{equation}\label{eq:faithful-test-specialization}
 c_{ij}=-2\delta_{ij},
 \qquad
 s_i=-2x_i,
 \qquad
 r=-\sum_{j=1}^M x_j^2.
\end{equation}
This particular specialization is sufficient: if \(D^Y_k\) were the zero element of \(\T_Y\), it would remain zero after \eqref{eq:faithful-test-specialization}.}

\newrev{Under this specialization the radial normal-form basis remains linearly independent. To see this, put \(z=x_M\). The induced scalar map
\[
 \phi:\R[r,s_1,\ldots,s_N]\longrightarrow
 \R[x_1,\ldots,x_N,z]
\]
sends \(s_i\) to \(-2x_i\) and
\[
 r\longmapsto -z^2-\frac14\sum_{i=1}^Ns_i^2.
\]
It is injective. Indeed, write a nonzero polynomial as
\[
 P(r,s)=\sum_{\ell=0}^d a_\ell(s)r^\ell,
 \qquad a_d(s)\ne0.
\]
After substituting \(r=-z^2-\frac14\sum_i s_i^2\), the coefficient of \(z^{2d}\) is \((-1)^d a_d(s)\), so the resulting polynomial cannot vanish identically. Thus no nonzero polynomial in \(r,s_1,\ldots,s_N\) is killed by \(\phi\).}

\newrev{Now suppose that a specialized radial normal form vanishes,
\begin{equation}\label{eq:faithful-normal-form}
 \sum_I e_I A_I(r,s)
 +\sum_I (x\wedge e_I)B_I(r,s)=0.
\end{equation}
For each \(I\subseteq\{1,\ldots,N\}\), the Clifford-basis coefficient of \(e_M\wedge e_I\) in \eqref{eq:faithful-normal-form} is, up to the fixed ordering sign, \(zB_I(r,s)\). No term with a different index set contributes to that basis blade. Since the polynomial ring is a domain and \(\phi\) is injective, every \(B_I\) is zero. The remaining relation \(\sum_Ie_IA_I=0\) then gives \(A_I=0\) for all \(I\). Consequently the specialized basis \eqref{eq:H-basis} is linearly independent, and the classical specialization of \(H^Y_k\) identifies with a finite-dimensional subspace
\[
 \widehat H^Y_k\subset \mathcal P_k(\R^M;\mathrm{Cl}_M)
\]
of homogeneous Clifford-valued polynomials.}

\newrev{The required nonsingularity is obtained from the ordinary ambient Fischer decomposition. For the left Dirac operator
\[
 D_LF=\sum_{j=1}^Me_j\frac{\partial F}{\partial x_j},
\]
the Euclidean Fischer decomposition reviewed in
\cite{BrackxDelangheSommen1982,BrackxDeSchepperSoucek2010} applies to homogeneous polynomial spinors. More precisely, the decomposition displayed in \cite{BrackxDeSchepperSoucek2010} writes the full polynomial space as the direct sum of the spaces \(x^p\mathcal M^L_j\). Taking its homogeneous component of degree \(k+1\) separates the term with \(p=0\) from all terms with \(p\ge1\), and therefore gives
\begin{equation}\label{eq:ambient-left-fischer}
 \mathcal P_{k+1}(\R^M;\mathrm{Cl}_M)
 =\mathcal M^L_{k+1}\oplus x\mathcal P_k(\R^M;\mathrm{Cl}_M),
 \qquad
 \mathcal M^L_{k+1}=\ker D_L.
\end{equation}
The 2010 reference formulates this decomposition for polynomials with values in an irreducible spinor module, so one further passage is needed for the Clifford-valued space used here. Complexify if necessary. The complex Clifford algebra, viewed as a left module over itself, is a finite direct sum of irreducible spinor modules; applying the cited decomposition to each summand gives the complexified Clifford-valued version of \eqref{eq:ambient-left-fischer}. For a real Clifford-valued polynomial, complex conjugation preserves both summands, and uniqueness of the complexified decomposition makes both components real. Thus \eqref{eq:ambient-left-fischer} holds on the real Clifford-valued polynomial space as well. Equation~\eqref{eq:ambient-left-fischer} is an ambient polynomial-space decomposition; no Fischer decomposition of the radial subspace is assumed.

Let \(\widetilde{\phantom F}\) denote Clifford reversion. Since \(\widetilde{x}=x\),
\[
 \widetilde{xF}=\widetilde F x,
 \qquad
 \widetilde{D_LF}=D_R(\widetilde F),
 \qquad
 D_RF:=\sum_{j=1}^M\frac{\partial F}{\partial x_j}e_j.
\]
Therefore reversion turns \eqref{eq:ambient-left-fischer} into the right-handed decomposition
\begin{equation}\label{eq:ambient-right-fischer}
 \mathcal P_{k+1}(\R^M;\mathrm{Cl}_M)
 =\mathcal M^R_{k+1}\oplus \mathcal P_k(\R^M;\mathrm{Cl}_M)x,
 \qquad
 \mathcal M^R_{k+1}=\ker D_R.
\end{equation}
Our convention is \(\partial_x^{\mathrm R}=-D_R\), so the two operators have the same kernel.}

\newrev{It follows directly from \eqref{eq:ambient-right-fischer} that
\[
 \partial_x^{\mathrm R}R_x:\mathcal P_k(\R^M;\mathrm{Cl}_M)
 \longrightarrow\mathcal P_k(\R^M;\mathrm{Cl}_M)
\]
is injective. Indeed, if \(\partial_x^{\mathrm R}(Ux)=0\), then \(Ux\) belongs to both summands in \eqref{eq:ambient-right-fischer}, hence \(Ux=0\). Multiplication once more on the right by \(x\) gives
\[
 0=(Ux)x=Ux^2=-|x|^2U,
\]
and therefore \(U=0\). In particular, its restriction to \(\widehat H^Y_k\) is injective. Proposition~\ref{prop:classical-limit-order} shows that this restriction preserves \(\widehat H^Y_k\), because it is the classical limit of the radial operator \(\Kop^Y_k\). Since \(\widehat H^Y_k\) is finite-dimensional, the determinant of the restricted classical operator in the specialized radial basis is nonzero.}

\newrev{Finally specialize \(Q=q^M\) and \(c_{ij}=-2\delta_{ij}\) in \(\Kop^Y_k\). Every matrix entry has a removable singularity at \(q=1\), and Proposition~\ref{prop:classical-limit-order}, together with the fact that right multiplication is undeformed, gives
\begin{equation}\label{eq:K-classical-limit}
 \lim_{q\to1}
 \left.\sigma_M(\Kop^Y_k)\right|_{c_{ij}=-2\delta_{ij}}
 =\partial_x^{\mathrm R}R_x\bigm|_{\widehat H^Y_k}.
\end{equation}
Since the limiting operator is injective on the finite-dimensional space \(\widehat H^Y_k\), its determinant is nonzero. Continuity of the determinant in the matrix entries therefore shows that
\[
 \left.\sigma_M(D^Y_k)\right|_{c_{ij}=-2\delta_{ij}}
\]
is not identically zero as a rational function of \(q\), contradicting \(D^Y_k=0\).}
\end{proof}

\newrev{Generic nonvanishing does not exclude zeros at particular numerical specializations. Section~6 determines these factors explicitly in the smallest supports and identifies a degree-zero factor for general support.}

\newrev{\begin{remark}
Proposition~\ref{prop:det-nonzero} uses the classical theory only to exhibit one nonsingular specialization. It does not assert that a given numerical pair \((q,M)\) is nonresonant. Numerical nonvanishing is a separate question addressed explicitly in Section~6 for low support and, in general, encoded by the determinant \(D^Y_k\).
\end{remark}}

\newrev{Once \(D^Y_k\) is inverted, the usual Fischer subtraction can be carried out explicitly on each homogeneous block. The formula below also identifies the monogenic projector.}

\begin{theorem}[Finite determinant-localized Fischer decomposition]\label{thm:finite-det-fischer}
After localizing \(\T_Y\) by \(D^Y_k\), one has
\begin{equation}\label{eq:finite-det-fischer}
 \newrev{(H^Y_{k+1})_{D^Y_k}
 =\ker\!\left(\dY:(H^Y_{k+1})_{D^Y_k}\to(H^Y_k)_{D^Y_k}\right)
 \oplus R_x(H^Y_k)_{D^Y_k}.}
\end{equation}
The projection onto the right \(q\)-monogenic summand is
\begin{equation}\label{eq:mon-projection}
 \newrev{\Pi^{Y}_{k+1}
 =I-R_x(\Kop^Y_k)^{-1}\dY.}
\end{equation}
\end{theorem}

\begin{proof}
Localization by \(D^Y_k\) makes \(\Kop^Y_k\) invertible. For \(F\in(H^Y_{k+1})_{D^Y_k}\), define
\[
 U=(\Kop^Y_k)^{-1}\dY F,
 \qquad
 F_0=F-\newrev{R_xU}=F-Ux.
\]
Then
\[
 \dY F_0
 =\dY F-\dY R_x(\Kop^Y_k)^{-1}\dY F
 =\dY F-\Kop^Y_k(\Kop^Y_k)^{-1}\dY F=0.
\]
Thus every \(F\) is a sum of a right \(q\)-monogenic element and an element of \(R_x(H^Y_k)\).

For directness, suppose \(R_xU\) also lies in the kernel. Then
\[
 0=\dY R_xU=\Kop^Y_kU,
\]
so invertibility of \(\Kop^Y_k\) gives \(U=0\). \newrev{Equivalently, the operator \(\Pi^Y_{k+1}\) satisfies
\[
 \dY\Pi^Y_{k+1}=0,
 \qquad
 \Pi^Y_{k+1}F=F\quad(F\in\ker\dY),
 \qquad
 \Pi^Y_{k+1}R_xU=0.
\]
These three identities identify it as the projection onto the first summand along the second. In particular \((\Pi^Y_{k+1})^2=\Pi^Y_{k+1}\).}
\end{proof}

\newrev{To pass from finite parameter sets to the ambient radial algebra, it remains to prove compatibility of the finite projectors under enlargement of the auxiliary set and to account explicitly for the determinant localization.}

\begin{theorem}[Global determinant-localized Fischer decomposition]\label{thm:global-mon-fischer}
Let \(\Omega_x^{\mathrm{mon}}\) be the central multiplicative set generated by all determinants \(D^Y_k\), where \(Y\) ranges over the finite subsets of \(S\setminus\{x\}\) and \(k\ge0\). Then
\begin{equation}\label{eq:global-mon-fischer}
 \newrev{\RB(S)_{\Omega_x^{\mathrm{mon}}}
 =\ker\!\left(\dglob:\RB(S)_{\Omega_x^{\mathrm{mon}}}
 \to\RB(S)_{\Omega_x^{\mathrm{mon}}}\right)
 \oplus R_x\RB(S)_{\Omega_x^{\mathrm{mon}}}.}
\end{equation}
Equivalently, the second summand is \(\RB(S)_{\Omega_x^{\mathrm{mon}}}x\).
\end{theorem}

\begin{proof}
\newrev{By Proposition~\ref{prop:det-nonzero}, every generator \(D^Y_k\) is nonzero. The central scalar algebra generated by the anticommutators is a polynomial algebra over the domain \(\B\), so finite products of these generators are nonzero; in particular, \(0\notin\Omega_x^{\mathrm{mon}}\). Every generator \(D^Y_k\) belongs to \(\T_Y\); hence it depends only on \(q,Q\) and anticommutators among auxiliary vectors, not on the distinguished vector \(x\). The defining formulas for \(\dglob\) carry all such central coefficients unchanged. Thus \(\dglob\) is linear over the central coefficient algebra containing \(\Omega_x^{\mathrm{mon}}\), and both \(\dglob\) and \(R_x\) extend to the localization.

Write
\[
 \mathcal R_x:=\RB(S)_{\Omega_x^{\mathrm{mon}}}.
\]
An arbitrary element of \(\mathcal R_x\) can be written as \(\omega^{-1}F\), with
\(F\in\RB(S)\) and \(\omega\in\Omega_x^{\mathrm{mon}}\). Since \(\omega\) is central and \(\dglob\omega=0\), a decomposition of \(F\) may be multiplied by \(\omega^{-1}\). It is therefore enough to construct the decomposition for an unlocalized numerator \(F\), while all inverses are taken in \(\mathcal R_x\).}

\newrev{Every element \(F\in\RB(S)\) involves only finitely many auxiliary vectors and is a finite sum of homogeneous \(x\)-degree components. Choose a finite set \(Y\) containing all auxiliary variables occurring in \(F\), and write \(F=\sum_nF_n\) with \(F_n\in H^Y_n\). Since \(\dY H^Y_0=0\), the degree-zero component is already right \(q\)-monogenic. For every \(n\ge1\), Theorem~\ref{thm:finite-det-fischer} with \(k=n-1\), followed by the further localization from \((\T_Y)_{D^Y_{n-1}}\) to \(\mathcal R_x\), gives the unique decomposition
\[
 F_n=(F_n)_0+R_xU_{n-1},
 \qquad
 \dY(F_n)_0=0.
\]
Set \((F_0)_0=F_0\). Summing over the finitely many degrees gives a relative Fischer decomposition of \(F\) in \(\mathcal R_x\), and hence also of \(\omega^{-1}F\).}

\newrev{We now verify that the decomposition is independent of the chosen finite auxiliary set. If \(Y\subset Z\), then right multiplication by \(x\) introduces no new auxiliary vector and Theorem~\ref{thm:direct-limit-derivative} gives
\[
 \Kop^Z_k\bigm|_{H^Y_k}=\Kop^Y_k.
\]
Thus \(H^Y_k\) is invariant under \(\Kop^Z_k\). After extension of scalars to \(\mathcal R_x\), the restriction \(\Kop^Y_k\) is invertible. For \(V\in H^Y_k\), the element \((\Kop^Y_k)^{-1}V\) is therefore a preimage of \(V\) under \(\Kop^Z_k\); uniqueness of the inverse of \(\Kop^Z_k\) gives
\[
 (\Kop^Z_k)^{-1}V=(\Kop^Y_k)^{-1}V.
\]
Consequently the projectors \(\Pi^Y_{k+1}\) and \(\Pi^Z_{k+1}\) agree on \(H^Y_{k+1}\), after the common localization. Together with the identity on degree zero, they therefore define a global projector
\[
 \Pi:\mathcal R_x\longrightarrow\ker\dglob
\]
degree by degree, with \(\Pi F=F\) on \(\ker\dglob\) and
\(\Pi(R_xU)=0\) for every \(U\in\mathcal R_x\). Hence
\[
 \mathcal R_x=\ker\dglob+R_x\mathcal R_x.
\]
If an element belongs to both summands, applying \(\Pi\) to it gives the element itself from the first description and \(0\) from the second. The intersection is therefore zero, proving \eqref{eq:global-mon-fischer}.}
\end{proof}

\section{Explicit determinant factors and exceptional specializations}

\newrev{This section studies the determinants occurring in Section~5. We first compute the scalar block with no auxiliary vector and then the complete matrix for one auxiliary vector, displaying the intermediate multiplication and differentiation formulas. For arbitrary finite support, an exact-support filtration gives a factorization of the determinant and an explicit degree-zero factor.}

Theorem~\ref{thm:global-mon-fischer} gives the general right \(q\)-monogenic Fischer decomposition after determinant localization. We identify the determinant factors for the smallest finite supports and then factor the general determinant by exact \(x\)-support. \newrev{All computations use the derivative \eqref{eq:derivative-x} together with right multiplication \eqref{eq:Rx-y}--\eqref{eq:Rx-x}.}

\subsection{One vector}

\newrev{With no auxiliary vector, each homogeneous block has rank one over the scalar coefficient ring. Right and left multiplication coincide, and the corresponding determinant can therefore be computed directly.}

For \(Y=\varnothing\), the homogeneous block is one-dimensional. Right and left multiplication by \(x\) coincide in this case, and the determinant is
\begin{equation}\label{eq:empty-factor}
 D^\varnothing_k=\begin{cases}
 \qnum{m+k},& k\text{ even},\\
 \qnum{k+1},& k\text{ odd}.
 \end{cases}
\end{equation}
Indeed, for \(k=2a\),
\[
 r^a\xmapsto{R_x}xr^a\xmapsto{\partial^{\mathrm R}_{x,q}}\qnum{m+2a}r^a,
\]
whereas for \(k=2a+1\),
\[
 xr^a\xmapsto{R_x}r^{a+1}\xmapsto{\partial^{\mathrm R}_{x,q}}\qnum{2a+2}xr^a.
\]

\newrev{\subsection{One auxiliary vector: complete two-vector determinant}}

\newrev{With one auxiliary vector, the exterior-blade term and right multiplication both contribute to the homogeneous matrices. We give the intermediate formulas for each matrix entry. The calculation splits according to Clifford parity and is block triangular with respect to the radial exponent.}

Let \(Y=\{y\}\) and put
\[
 r=x^2,
 \qquad s=\{x,y\},
 \qquad t=y^2,
 \qquad w=x\wedge y.
\]
\newrev{Since \(c_{11}=\{y,y\}=2t\), the coefficient ring
\(\T_{\{y\}}=\B[c_{11}]\) is equivalently written as \(\B[t]\).}
The elementary products needed below are
\begin{equation}\label{eq:two-vector-products}
xy=w+\frac{s}{2},\qquad yx=-w+\frac{s}{2},\qquad
wx=\frac{s}{2}x-ry,\qquad wy=tx-\frac{s}{2}y.
\end{equation}
\newrev{The first and third identities are exactly the \(p=1\) instances of the left exterior product and the right-multiplication formula in Lemma~\ref{lem:Rx-formulas}; the second is \(R_x(y)\).}

\newrev{It is convenient to record four scalar-coefficient differentiation formulas that will be used repeatedly. For \(f=f(r,s,t)\), write
\[
 \dscal f=x\alpha_f+y\beta_f,
 \qquad
 \alpha_f=(1+q)\delta^{(r)}_{q^2}f,
 \qquad
 \beta_f=2T^{(r)}_{q^2}\delta^{(s)}_qf.
\]
Using \(x^2=r\), \(xy=w+s/2\), \(yx=-w+s/2\), \(wx=(s/2)x-ry\), and \(wy=tx-(s/2)y\), Definition~\ref{def:finite-derivative} gives
\begin{align}
 \partial^{\{y\},\mathrm R}_{x,q}(f)
 &=x\alpha_f+y\beta_f,\label{eq:d-scalar-helper}\\
 \partial^{\{y\},\mathrm R}_{x,q}(xf)
 &=\qnum m f+Q\left(r\alpha_f+\left(w+\frac{s}{2}\right)\beta_f\right),\label{eq:d-x-helper}\\
 \partial^{\{y\},\mathrm R}_{x,q}(yf)
 &=\left(-w+\frac{s}{2}\right)\alpha_f+t\beta_f,\label{eq:d-y-helper}\\
 \partial^{\{y\},\mathrm R}_{x,q}(wf)
 &=-\qnum{m-1}yf
 +Qq^{-1}\left[\left(\frac{s}{2}x-ry\right)\alpha_f
 +\left(tx-\frac{s}{2}y\right)\beta_f\right].\label{eq:d-w-helper}
\end{align}
For the monomial \(f=r^as^b\), these coefficients reduce to
\[
 \alpha_f=\qnum{2a}r^{a-1}s^b,
 \qquad
 \beta_f=2q^{2a}\qnum b\,r^as^{b-1},
\]
with a term omitted when its exponent is zero. These identities reduce the determinant calculation to scalar \(q\)-number formulas.}

For later use define
\begin{align}
 A_{k,a}&:=\qnum m+Q\left(\qnum{2a}+2q^{2a}\qnum{k-2a}\right),
 &&0\le a\le\left\lfloor\frac{k-1}{2}\right\rfloor,
 \label{eq:Aka}\\
 \newrev{B^{\mathrm R}_{k,a}}
 &\newrev{:=\qnum{m-1}+Qq^{-1}\left(\qnum{k}+q^{2a+2}\qnum{k-1-2a}\right),}
 &&0\le a\le\left\lfloor\frac{k-2}{2}\right\rfloor,
 \label{eq:Bka}\\
 \newrev{C^{\mathrm R}_k}
 &\newrev{:=\qnum{m-1}+\qnum{k+1}+Qq^{-1}\qnum{k}.}
 \label{eq:Ck}
\end{align}
Empty products below are interpreted as \(1\).

\newrev{The following proposition gives the closed determinant. Its proof derives each diagonal block from \eqref{eq:d-scalar-helper}--\eqref{eq:d-w-helper} before taking determinants.}

\begin{theorem}[Two-vector determinant]\label{thm:two-vector-det}
For every \(k\ge0\),
\newrev{\begin{equation}\label{eq:closed-two-vector}
\begin{aligned}
D^{\{y\}}_k
&=\bigl(\qnum m+Q\qnum{k}\bigr)C^{\mathrm R}_k
  \prod_{a=0}^{k/2-1}\qnum{2a+2}A_{k,a}
  \prod_{a=0}^{k/2-1}\qnum{2a+2}B^{\mathrm R}_{k,a},
  &&k\text{ even},\\[2mm]
D^{\{y\}}_k
&=\qnum{k+1}C^{\mathrm R}_k
  \prod_{a=0}^{(k-1)/2}\qnum{2a+2}A_{k,a}
  \prod_{a=0}^{(k-3)/2}\qnum{2a+2}B^{\mathrm R}_{k,a},
  &&k\text{ odd}.
\end{aligned}
\end{equation}}
\end{theorem}

\begin{proof}
\newrev{The parity decomposition
\[
 \mathcal E=\B[r,s,t]\oplus w\B[r,s,t],
 \qquad
 \mathcal O=x\B[r,s,t]\oplus y\B[r,s,t]
\]
is preserved by \(\Kop_k^{\{y\}}=\partial^{\{y\},\mathrm R}_{x,q}R_x\). We compute the even and odd Clifford-parity determinant blocks separately. These parity spaces are denoted by \(\mathcal E,\mathcal O\), reserving \(\mathcal P_k(\R^M;\mathrm{Cl}_M)\) for the homogeneous polynomial spaces used in Proposition~\ref{prop:det-nonzero}.}

\paragraph{The even Clifford-parity block.}
In \(\mathcal E\cap H_k^{\{y\}}\), set
\[
 u_a=r^as^{k-2a},
 \qquad
 v_a=wr^as^{k-1-2a}.
\]
The vectors \(u_a\) exist for \(0\le a\le\lfloor k/2\rfloor\), and the \(v_a\) for \(0\le a\le\lfloor(k-1)/2\rfloor\). Order the basis as
\[
 (u_0,v_0),(u_1,v_1),\ldots,
\]
with the final unpaired \(u_{k/2}\) appended when \(k\) is even.

\newrev{Put
\[
 b:=k-2a,
 \qquad
 C_{k,a}:=\qnum m+Q\left(\qnum{2a}+q^{2a}\qnum b\right).
\]
For \(u_a=r^as^b\), right multiplication gives \(R_xu_a=xr^as^b\). For \(f=r^as^b\), the monomial rule gives
\[
 \alpha_f=\qnum{2a}r^{a-1}s^b,
 \qquad
 \beta_f=2q^{2a}\qnum b\,r^as^{b-1}.
\]
Substituting these into \eqref{eq:d-x-helper}, and using \((w+s/2)s^{b-1}=ws^{b-1}+s^b/2\), yields
\begin{align*}
 \Kop_k^{\{y\}}u_a
 &=\qnum m u_a+Q\qnum{2a}u_a
   +Qq^{2a}\qnum b\,u_a
   +2Qq^{2a}\qnum b\,v_a\\
 &=C_{k,a}u_a+2Qq^{2a}\qnum b\,v_a.
\end{align*}
Thus
\begin{equation}\label{eq:Ku}
\Kop_k^{\{y\}}u_a
=C_{k,a}u_a+2Qq^{2a}\qnum{k-2a}v_a.
\end{equation}

For \(v_a=wr^as^{b-1}\), the right-multiplication identity \(wx=(s/2)x-ry\) gives
\[
 R_xv_a=x\left(\frac12r^as^b\right)
        -y\left(r^{a+1}s^{b-1}\right).
\]
Set \(f_1=\frac12r^as^b\) and \(f_2=r^{a+1}s^{b-1}\). Their Jackson coefficients are
\[
 \alpha_{f_1}=\frac12\qnum{2a}r^{a-1}s^b,
 \quad
 \beta_{f_1}=q^{2a}\qnum b\,r^as^{b-1},
\]
\[
 \alpha_{f_2}=\qnum{2a+2}r^as^{b-1},
 \quad
 \beta_{f_2}=2q^{2a+2}\qnum{b-1}r^{a+1}s^{b-2}.
\]
Using \eqref{eq:d-x-helper} for \(xf_1\), we obtain
\[
 \partial^{\{y\},\mathrm R}_{x,q}(xf_1)
 =\frac12C_{k,a}u_a+Qq^{2a}\qnum b\,v_a.
\]
Using \eqref{eq:d-y-helper} for \(yf_2\), we obtain
\[
 \partial^{\{y\},\mathrm R}_{x,q}(yf_2)
 =\frac12\qnum{2a+2}u_a
  -\qnum{2a+2}v_a
  +2q^{2a+2}\qnum{b-1}t\,u_{a+1}.
\]
Subtracting the last two displays gives
\begin{align}
\Kop_k^{\{y\}}v_a
&=\frac12\left(C_{k,a}-\qnum{2a+2}\right)u_a
 +\left(\qnum{2a+2}+Qq^{2a}\qnum{k-2a}\right)v_a\notag\\
&\quad-2q^{2a+2}\qnum{k-1-2a}\,t\,u_{a+1}.
\label{eq:Kv}
\end{align}
The last term of \eqref{eq:Kv} belongs to the next pair \((u_{a+1},v_{a+1})\) and is therefore strictly block triangular. With columns representing the input vectors \((u_a,v_a)\) and rows their output coefficients, the diagonal block is
\begin{equation}\label{eq:even-block}
E^{\mathrm R}_{k,a}=
\begin{pmatrix}
C_{k,a}
&\frac12\left(C_{k,a}-\qnum{2a+2}\right)\\[1mm]
2Qq^{2a}\qnum{k-2a}
&\qnum{2a+2}+Qq^{2a}\qnum{k-2a}
\end{pmatrix}.
\end{equation}}

\newrev{A direct calculation, retaining the column/row convention just stated, gives
\[
\begin{aligned}
\det E^{\mathrm R}_{k,a}
&=C_{k,a}\left(\qnum{2a+2}+Qq^{2a}\qnum{k-2a}\right)\\
&\quad-Qq^{2a}\qnum{k-2a}
 \left(C_{k,a}-\qnum{2a+2}\right)\\
&=\qnum{2a+2}
 \left(\qnum m+Q\left(\qnum{2a}+2q^{2a}\qnum{k-2a}\right)\right)\\
&=\qnum{2a+2}A_{k,a}.
\end{aligned}
\]
If \(k\) is even, the unpaired vector \(u_{k/2}=r^{k/2}\) contributes \(\qnum m+Q\qnum{k}\). Hence
\begin{equation}\label{eq:even-det}
\det(\Kop_k^{\{y\}}|_{\mathcal E\cap H_k^{\{y\}}})
=\begin{cases}
\displaystyle
\bigl(\qnum m+Q\qnum{k}\bigr)
\prod_{a=0}^{k/2-1}\qnum{2a+2}A_{k,a},&k\text{ even},\\[2mm]
\displaystyle
\prod_{a=0}^{(k-1)/2}\qnum{2a+2}A_{k,a},&k\text{ odd}.
\end{cases}
\end{equation}}

\paragraph{The odd Clifford-parity block.}
Put
\[
 \xi_a=xr^as^{k-1-2a},
 \qquad
 \upsilon_a=yr^as^{k-2a}.
\]
The vector \(\upsilon_0\) is unpaired; for \(0\le a\le\lfloor(k-2)/2\rfloor\), pair \(\xi_a\) with \(\upsilon_{a+1}\). \newrev{We order this parity block as
\[
 \upsilon_0,\;(\xi_0,\upsilon_1),\;(\xi_1,\upsilon_2),\ldots,
\]
with the final unpaired \(\xi_{(k-1)/2}\) appended when \(k\) is odd.}
This order makes the \(t\xi_{a+1}\) terms below strictly block triangular.

\newrev{For \(\xi_a=xr^as^{k-1-2a}\), associativity gives the particularly simple right product
\[
 R_x(\xi_a)=r^{a+1}s^{k-1-2a}.
\]
If \(b=k-1-2a\), the scalar rule \eqref{eq:d-scalar-helper} has
\[
 \alpha=\qnum{2a+2}r^as^b,
 \qquad
 \beta=2q^{2a+2}\qnum b\,r^{a+1}s^{b-1},
\]
so
\begin{equation}\label{eq:Kxi}
\Kop_k^{\{y\}}\xi_a
=\qnum{2a+2}\xi_a
 +2q^{2a+2}\qnum{k-1-2a}\upsilon_{a+1}.
\end{equation}

For \(\upsilon_a=yr^as^b\) with now \(b=k-2a\), right multiplication gives
\[
 R_x(\upsilon_a)=\left(-w+\frac{s}{2}\right)r^as^b
 =-wr^as^b+\frac12r^as^{b+1}.
\]
For \(f=r^as^b\), use
\[
 \alpha_f=\qnum{2a}r^{a-1}s^b,
 \qquad
 \beta_f=2q^{2a}\qnum b\,r^as^{b-1}.
\]
Equation~\eqref{eq:d-w-helper} gives
\begin{align*}
-\partial^{\{y\},\mathrm R}_{x,q}(wf)
&=-\frac12Qq^{-1}\qnum{2a}\,\xi_{a-1}
  -2Qq^{2a-1}\qnum b\,t\,\xi_a\\
&\quad+
 \left[\qnum{m-1}+Qq^{-1}
 \left(\qnum{2a}+q^{2a}\qnum b\right)\right]\upsilon_a.
\end{align*}
On the other hand, for \(h=\frac12r^as^{b+1}\), equation~\eqref{eq:d-scalar-helper} gives
\[
 \partial^{\{y\},\mathrm R}_{x,q}(h)
 =\frac12\qnum{2a}\,\xi_{a-1}
  +q^{2a}\qnum{b+1}\upsilon_a.
\]
Adding these two contributions yields
\begin{align}
\Kop_k^{\{y\}}\upsilon_a
&=\frac12(1-Qq^{-1})\qnum{2a}\,\xi_{a-1}
 -2Qq^{2a-1}\qnum{k-2a}\,t\,\xi_a\notag\\
&\quad+
 \left[
 \qnum{m-1}
 +Qq^{-1}\left(\qnum{2a}+q^{2a}\qnum{k-2a}\right)
 +q^{2a}\qnum{k-2a+1}
 \right]\upsilon_a.
\label{eq:Kupsilon}
\end{align}
When \eqref{eq:Kupsilon} is applied to the paired input \(\upsilon_{a+1}\), its term involving \(t\xi_{a+1}\) lies beyond the current pair and is strictly block triangular. Therefore, for the ordered pair \((\xi_a,\upsilon_{a+1})\), with columns as inputs and rows as output coefficients, the diagonal block is
\begin{equation}\label{eq:odd-block}
O^{\mathrm R}_{k,a}=
\begin{pmatrix}
\qnum{2a+2}
&\frac12(1-Qq^{-1})\qnum{2a+2}\\[1mm]
2q^{2a+2}\qnum{k-1-2a}
&\Gamma_{k,a}
\end{pmatrix},
\end{equation}
where
\begin{equation}\label{eq:Gamma}
\Gamma_{k,a}
=\qnum{m-1}
 +Qq^{-1}\left(\qnum{2a+2}+q^{2a+2}\qnum{k-2a-2}\right)
 +q^{2a+2}\qnum{k-2a-1}.
\end{equation}}

\newrev{With this column convention, the upper-right entry is the coefficient of \(\xi_a\) in \(\Kop\upsilon_{a+1}\), while the lower-left entry is the coefficient of \(\upsilon_{a+1}\) in \(\Kop\xi_a\). The determinant is
\[
\begin{aligned}
\det O^{\mathrm R}_{k,a}
&=\qnum{2a+2}
 \left(\Gamma_{k,a}-(1-Qq^{-1})q^{2a+2}\qnum{k-1-2a}\right)\\
&=\qnum{2a+2}
 \left[\qnum{m-1}+Qq^{-1}
 \left(\qnum{k}+q^{2a+2}\qnum{k-1-2a}\right)\right]\\
&=\qnum{2a+2}B^{\mathrm R}_{k,a}.
\end{aligned}
\]
In the second line we used the \(q\)-number identities
\[
 \qnum{2a+2}+q^{2a+2}\qnum{k-2a-2}=\qnum{k},
\]
and
\[
 \qnum{2a+2}+q^{2a+2}\qnum{k-1-2a}=\qnum{k+1}.
\]
The unpaired vector \(\upsilon_0=ys^k\) has diagonal coefficient
\[
 C^{\mathrm R}_k=\qnum{m-1}+\qnum{k+1}+Qq^{-1}\qnum{k},
\]
while, when \(k\) is odd, the final unpaired \(\xi_{(k-1)/2}\) contributes \(\qnum{k+1}\). Consequently
\begin{equation}\label{eq:odd-det}
\det(\Kop_k^{\{y\}}|_{\mathcal O\cap H_k^{\{y\}}})
=C^{\mathrm R}_k
\begin{cases}
\displaystyle
\prod_{a=0}^{k/2-1}\qnum{2a+2}B^{\mathrm R}_{k,a},&k\text{ even},\\[2mm]
\displaystyle
\qnum{k+1}\prod_{a=0}^{(k-3)/2}\qnum{2a+2}B^{\mathrm R}_{k,a},&k\text{ odd}.
\end{cases}
\end{equation}}
Multiplying \eqref{eq:even-det} and \eqref{eq:odd-det} gives \eqref{eq:closed-two-vector}.
\end{proof}

\newrev{For \(k=0\), the \(\T_{\{y\}}\)-basis of \(H^{\{y\}}_0\) is \((1,y)\), and the two computations
\[
 \Kop^{\{y\}}_0(1)=\partial^{\{y\},\mathrm R}_{x,q}(x)=\qnum m,
 \qquad
 \Kop^{\{y\}}_0(y)=\bigl(\qnum{m-1}+1\bigr)y
\]
show directly that
\[
 D^{\{y\}}_0=\qnum m\bigl(\qnum{m-1}+1\bigr).
\]
Formula~\eqref{eq:closed-two-vector} gives exactly the same value, because all products are empty and \(C^{\mathrm R}_0=\qnum{m-1}+1\).}

\newrev{After the specialization \(Q=q^M\), the factors in the determinant have definite signs for \(0<q<1\).}

\begin{corollary}[Two-vector nonresonance]\newrev{\label{cor:two-vector-stable}}
\newrev{Assume \(0<q<1\), specialize \(Q=q^M\), and let \(M\ge1\) be an integer. Then for every \(k\ge0\),
\begin{equation}\label{eq:two-vector-unlocalized}
 H^{\{y\}}_{k+1}
 =\ker\!\left(\partial^{\{y\},\mathrm R}_{x,q}:H^{\{y\}}_{k+1}\to H^{\{y\}}_k\right)
 \oplus R_xH^{\{y\}}_k
\end{equation}
holds without determinant localization.}
\end{corollary}

\begin{proof}
\newrev{For \(0<q<1\), every ordinary \(q\)-number \(\qnum n\) with positive integer \(n\) is positive. After \(Q=q^M\),
\[
 \qnum m=\qnum M>0,
 \qquad
 \qnum{m-1}=\qnum{M-1}\ge0,
 \qquad
 Qq^{-1}=q^{M-1}>0.
\]
The factors \(\qnum{2a+2}\), \(\qnum{k+1}\), and \(\qnum M+q^M\qnum{k}\) are therefore positive. Furthermore,
\[
 A_{k,a}=\qnum M+q^M
 \left(\qnum{2a}+2q^{2a}\qnum{k-2a}\right)>0,
\]
and whenever \(B^{\mathrm R}_{k,a}\) occurs,
\[
 B^{\mathrm R}_{k,a}
 =\qnum{M-1}+q^{M-1}
 \left(\qnum{k}+q^{2a+2}\qnum{k-1-2a}\right)>0,
\]
because \(k\ge2\) on that index range. Finally,
\[
 C^{\mathrm R}_k
 =\qnum{M-1}+\qnum{k+1}+q^{M-1}\qnum{k}>0.
\]
Thus every factor in \eqref{eq:closed-two-vector} is strictly positive. Moreover, the closed formula contains no \(t\), so after fixing \(q\) and \(M\) the determinant \(D_k^{\{y\}}\) is a positive real scalar and hence a unit of the specialized coefficient ring \(\R[t]\). The finite Fischer theorem therefore applies without determinant localization.}
\end{proof}

\newrev{Thus, for one auxiliary vector, the homogeneous determinant has no zeros for \(0<q<1\) after any positive integer dimension specialization. Larger supports are treated by the filtration below.}

\subsection{Support factorization}

\newrev{Closed matrices become unwieldy when several auxiliary vectors are present. Instead of attempting a full determinant formula, we filter by the set of auxiliary indices that occur either exteriorly or through a mixed scalar \(s_i\). The key point is that neither right multiplication nor the derivative can create a new such index.}

Let \(Y=\{y_1,\ldots,y_N\}\). For a basis monomial
\(y_Ir^as^\beta\) or \((x\wedge y_I)r^as^\beta\), define its \(x\)-support by
\begin{equation}\label{eq:xsupp}
 \supp_x(I,\beta):=I\cup\{i:\beta_i>0\}.
\end{equation}
The coefficient variables \(c_{ij}\) are not counted. Let \(\mathcal F^J_k(Y)\) be the \(\T_Y\)-span of the degree-\(k\) basis monomials whose \(x\)-support is contained in \(J\).

\newrev{Preservation of the support filtration gives a block-triangular matrix with respect to support inclusion, and hence a factorization into exact-support determinants.}

\begin{proposition}[Support factorization]\label{prop:support-factor}
The operator \(\Kop^Y_k\) preserves every \(\mathcal F^J_k(Y)\). It therefore induces an endomorphism on
\[
 \gr^JH^Y_k
 :=\mathcal F^J_k(Y)\Big/\sum_{J'\subsetneq J}\mathcal F^{J'}_k(Y).
\]
If \(D^Y_{k,J}:=\det(\gr^J\Kop^Y_k)\), then
\begin{equation}\label{eq:support-factor}
 D^Y_k=\prod_{J\subseteq\{1,\ldots,N\}}D^Y_{k,J}.
\end{equation}
Furthermore, \(D^Y_{k,J}\) is obtained by coefficient extension from the exact-support determinant for the smaller radial algebra generated by \(x\) and \(\{y_j:j\in J\}\).
\end{proposition}

\begin{proof}
\newrev{The right-multiplication formulas \eqref{eq:Rx-y}--\eqref{eq:Rx-x} do not increase \(x\)-support. The exterior term \(x\wedge y_I\) retains the same set \(I\); a contraction removes an exterior index \(i\) but inserts the corresponding mixed scalar \(s_i\), so the union defining \(\supp_x\) does not grow. Multiplication by the central variable \(r\) also leaves support unchanged.}

The derivative does not increase support either. The radial \(r\)-difference changes no auxiliary index, while differentiating in \(s_i\) removes one occurrence of \(s_i\) and inserts the vector \(y_i\). Thus \(\Kop^Y_k=\dY R_x\) preserves every \(\mathcal F^J_k(Y)\).

\newrev{Let \(E^J_k(Y)\) be the \(\T_Y\)-span of basis monomials with \(x\)-support exactly \(J\). The normal basis gives the direct sum
\[
 H^Y_k=\bigoplus_{J\subseteq\{1,\ldots,N\}}E^J_k(Y),
 \qquad
 \mathcal F^J_k(Y)=\bigoplus_{J'\subseteq J}E^{J'}_k(Y).
\]
Since every \(\mathcal F^J_k(Y)\) is invariant, an input from \(E^J_k(Y)\) can have output only in \(E^{J'}_k(Y)\) with \(J'\subseteq J\). Order the exact-support summands by nondecreasing cardinality, with any fixed order among supports of the same cardinality. The matrix of \(\Kop^Y_k\) is then block triangular, and its diagonal block on \(E^J_k(Y)\) represents the induced map on \(\gr^JH^Y_k\). Taking determinants proves \eqref{eq:support-factor}.}

On the exact-support quotient for \(J\), only the exterior variables \(y_j\) and the mixed scalars \(s_j\) with \(j\in J\) occur explicitly. Variables outside \(J\) can occur only through coefficients in \(\T_Y\), and all formulas carry those central coefficients along. \newrev{After quotienting by lower supports, contractions can introduce only coefficients \(c_{ij}\) with \(i,j\in J\).} Hence the exact-support operator is the scalar extension of the corresponding operator for the smaller set.
\end{proof}

\newrev{In degree zero, each nonempty exact-support quotient is one-dimensional, so its determinant can be computed directly.}

\begin{proposition}[Degree-zero exact-support factor]\label{prop:degree-zero-factor}
Let \(J\ne\varnothing\) and \(p=|J|\). In degree \(k=0\), the exact-support quotient is one-dimensional and
\newrev{\begin{equation}\label{eq:degree-zero-factor}
 D^Y_{0,J}=\qnum{m-p}+p.
\end{equation}}
For \(J=\varnothing\), \(D^Y_{0,\varnothing}=\qnum m\).
\end{proposition}

\begin{proof}
The empty-support statement is \eqref{eq:empty-factor}. Let
\(J=\{i_1<\cdots<i_p\}\ne\varnothing\). The only degree-zero exact-support basis element is \(y_J\). From \eqref{eq:Rx-y},
\newrev{\[
 R_xy_J=(-1)^p\left(
 x\wedge y_J-
 \sum_{\nu=1}^p(-1)^{\nu-1}\frac{s_{i_\nu}}2
 y_{J\setminus\{i_\nu\}}
 \right).
\]}

\newrev{Apply \(\dY\) and work modulo lower-support terms. For the first term, \eqref{eq:derivative-x} gives
\[
 (-1)^p\cdot(-1)^p\qnum{m-p}y_J
 =\qnum{m-p}y_J.
\]
For the \(\nu\)-th contraction term, the coefficient from right multiplication is
\((-1)^{p+1}(-1)^{\nu-1}\). Since
\(\dscal(s_{i_\nu}/2)=y_{i_\nu}\), its derivative is
\[
 (-1)^{p+1}(-1)^{\nu-1}
 y_{J\setminus\{i_\nu\}}y_{i_\nu}.
\]
Modulo lower support, only the exterior product survives. Moving \(y_{i_\nu}\) from the right end to its ordered position contributes \((-1)^{p-\nu}\), so the total sign is
\[
 (-1)^{p+1+\nu-1+p-\nu}=1.
\]
Thus each of the \(p\) contraction terms contributes \(+y_J\) in the exact-support quotient. Summing the blade term and all contractions gives \((\qnum{m-p}+p)y_J\), proving \eqref{eq:degree-zero-factor}.}
\end{proof}

\newrev{The sign cancellation in the preceding proof is the higher-support analogue of the elementary identity
\[
 \partial^{\{y\},\mathrm R}_{x,q}(yx)=\bigl(\qnum{m-1}+1\bigr)y.
\]
The Fischer side is essential here. If the corrected right derivative were instead paired with left multiplication by \(x\), the same exact-support calculation would give
\[
 (-1)^p\bigl(\qnum{m-p}-p\bigr).
\]
For the right pairing \(R_xF=Fx\) used in this paper, the additional signs from moving the final \(x\) through the exterior blade cancel the parity factor and give \(\qnum{m-p}+p\). In particular, the resulting degree-zero factor has no dependence on the parity of \(p\).}

\newrev{After the specialization \(Q=q^M\), the sign of the degree-zero factor can be determined explicitly. It is positive when \(p\le M\) and has one zero in \((0,1)\) when \(p>M\).}

\begin{corollary}[Degree-zero numerical exceptional set]\newrev{\label{cor:degree-zero-resonance}}
\newrev{Let \(M\ge1\) be an integer, specialize \(Q=q^M\), and assume \(0<q<1\). For an exact support of cardinality \(p\ge1\), the degree-zero factor is
\begin{equation}\label{eq:resonance-factor}
 f_{M,p}(q)=\qnum{M-p}+p.
\end{equation}
If \(p\le M\), then \(f_{M,p}(q)>0\) for every \(q\in(0,1)\). If \(p>M\), then \(f_{M,p}\) has exactly one zero \(q_0\in(0,1)\). For the abstract radial algebra with this coefficient specialization, the unlocalized right \(q\)-monogenic Fischer direct sum fails in degree one.}
\end{corollary}

\begin{proof}
\newrev{If \(p\le M\), then \(M-p\ge0\), so \(\qnum{M-p}\ge0\) and \(f_{M,p}(q)\ge p>0\).}

\newrev{Now suppose \(p>M\) and set \(d=p-M\ge1\). Using \(\qnum{-d}=-q^{-d}\qnum d\),
\[
 f_{M,p}(q)=p-q^{-d}\qnum d
 =p-\sum_{j=1}^dq^{-j}.
\]
The function \(g(q)=\sum_{j=1}^dq^{-j}\) is continuous and strictly decreasing on \((0,1)\), with \(g(q)\to\infty\) as \(q\to0^+\) and \(g(q)\to d=p-M<p\) as \(q\to1^-\). Hence there is a unique \(q_0\in(0,1)\) for which \(g(q_0)=p\), equivalently \(f_{M,p}(q_0)=0\).}

\newrev{By Proposition~\ref{prop:support-factor}, this exact-support factor divides the full degree-zero determinant. At \(q=q_0\), the specialized operator \(\Kop^Y_0\) is singular over the fraction field of the internal scalar coefficients. Choose a nonzero \(U\) in its kernel, clearing central denominators if necessary. Lemma~\ref{lem:Rx-injective} gives \(R_xU=Ux\ne0\), while
\[
 \partial^{Y,\mathrm R}_{x,q_0}(R_xU)=\Kop^Y_0U=0.
\]
Thus the kernel of the derivative has nonzero intersection with \(R_xH^Y_0\), and the direct sum fails.}
\end{proof}

\newrev{The case \(p>M\) lies outside the range in which \(p\) independent exterior directions can be represented faithfully in dimension \(M\).}

\newrev{\begin{remark}[Interpretation of the high-support resonance]
When \(p>M\), a radial algebra on \(p\) auxiliary vectors is not, in general, faithfully represented by \(p\) independent exterior directions in an \(M\)-dimensional Clifford space. Thus the zero in Corollary~\ref{cor:degree-zero-resonance} concerns the abstract radial algebra after the coefficient specialization \(Q=q^M\). In the range \(p\le M\), the degree-zero factor is strictly positive. Higher-degree exact-support determinants may have additional exceptional sets, which are retained in the general theorem through determinant localization.
\end{remark}}

\section*{Conclusion}

We have defined an intrinsic right \(q\)-radial vector derivative on arbitrary radial algebras by means of compatible finite relative operators. \newrev{After \(Q=q^M\), its classical limit agrees in every \(M\)-dimensional Clifford realization with the standard right Clifford derivative.} The anticommutator with exterior creation is triangular and yields an explicit Green decomposition after localization at its diagonal factors.

\newrev{Pairing the derivative with right multiplication \(R_xF=Fx\) gives determinant-localized Fischer decompositions for right \(q\)-monogenic elements, together with explicit homogeneous projectors. The determinants are explicit for zero or one auxiliary vector; in the latter case they are nonzero for every \(0<q<1\) and every positive integer \(M\). For arbitrary finite support, the determinant factors through exact-support quotients, with degree-zero rank-\(p\) factor \([m-p]_q+p\).}

\subsection*{Acknowledgments}
\newrev{The authors sincerely thank the editor and the reviewers for their careful reading of the manuscript and for their many helpful comments and suggestions, which have contributed substantially to improving the paper.}

\subsection*{Funding}
This work was by the University of Ostrava, Grant No.~SGS05/P\v{R}F/2026. \newrev{Yifan Zhang is also co-funded by the European Union under the REFRESH -- Research Excellence For REgion Sustainability and High-tech Industries project, No.~CZ.10.03.01/00/22003/0000048, via the Operational Programme Just Transition.}

\subsection*{Conflict of interest}
The authors declare that they have no competing interests related to the publication of this paper.

\subsection*{Authors' contributions}
All authors contributed equally to the manuscript and approved the final version.

\bibliographystyle{plainnat}
\bibliography{ref}

@book{BrackxDelangheSommen1982,
  author    = {Brackx, Fred and Delanghe, Richard and Sommen, Frank},
  title     = {Clifford Analysis},
  series    = {Research Notes in Mathematics},
  volume    = {76},
  publisher = {Pitman},
  address   = {Boston},
  year      = {1982}
}

@article{BrackxDeSchepperSoucek2010,
  author  = {Brackx, Fred and De Schepper, Hennie and Sou{\v{c}}ek, Vladim{\'i}r},
  title   = {Fischer Decompositions in Euclidean and Hermitean {Clifford} Analysis},
  journal = {Archivum Mathematicum},
  volume  = {46},
  number  = {5},
  pages   = {301--321},
  year    = {2010},
  url     = {https://dml.cz/handle/10338.dmlcz/141385}
}

@incollection{BernsteinZimmermannSchneider2026,
  author    = {Bernstein, Swanhild and Zimmermann, Martha Lina and Schneider, Baruch},
  title     = {The {$q$}-{Dirac} Operator on {Quantum Euclidean Space}},
  booktitle = {Schur Analysis and Applications to Hypercomplex Analysis, Neural Networks, and Linear Systems},
  editor    = {Alpay, Daniel and Lewkowicz, Izchak and Vajiac, Adrian and Vajiac, Mihaela},
  series    = {Operator Theory: Advances and Applications},
  volume    = {308},
  pages     = {99--118},
  publisher = {Springer},
  year      = {2026},
  doi       = {10.1007/978-3-032-02315-5_4}
}

@article{CoulembierSommen2010,
  author  = {Coulembier, Kevin and Sommen, Frank},
  title   = {{$q$}-Deformed Harmonic and {Clifford} Analysis and the {$q$}-{Hermite} and {Laguerre} Polynomials},
  journal = {Journal of Physics A: Mathematical and Theoretical},
  volume  = {43},
  number  = {11},
  pages   = {115202},
  year    = {2010},
  doi     = {10.1088/1751-8113/43/11/115202}
}

@article{CoulembierSommen2011,
  author  = {Coulembier, Kevin and Sommen, Frank},
  title   = {Operator Identities in {$q$}-Deformed {Clifford} Analysis},
  journal = {Advances in Applied Clifford Algebras},
  volume  = {21},
  number  = {4},
  pages   = {677--696},
  year    = {2011},
  doi     = {10.1007/s00006-011-0281-9}
}

@article{DeSchepperGuzmanAdanSommen2017,
  author  = {De Schepper, Hennie and Guzm{\'a}n Ad{\'a}n, Al{\'i} and Sommen, Frank},
  title   = {The Radial Algebra as an Abstract Framework for Orthogonal and {Hermitian Clifford} Analysis},
  journal = {Complex Analysis and Operator Theory},
  volume  = {11},
  pages   = {1139--1172},
  year    = {2017},
  doi     = {10.1007/s11785-016-0621-9}
}

@phdthesis{GuzmanAdan2018,
  author = {Guzm{\'a}n Ad{\'a}n, Al{\'i}},
  title  = {Euclidean and {Hermitian Clifford} Analysis on Superspace},
  school = {Ghent University},
  year   = {2018}
}

@book{KacCheung2002,
  author    = {Kac, Victor and Cheung, Pokman},
  title     = {Quantum Calculus},
  series    = {Universitext},
  publisher = {Springer},
  address   = {New York},
  year      = {2002},
  doi       = {10.1007/978-1-4613-0071-7}
}

@article{SabadiniStruppaSommenVanLancker2002,
  author  = {Sabadini, Irene and Struppa, Daniele C. and Sommen, Frank and Van Lancker, Peter},
  title   = {Complexes of {Dirac} Operators in {Clifford} Algebras},
  journal = {Mathematische Zeitschrift},
  volume  = {239},
  pages   = {293--320},
  year    = {2002},
  doi     = {10.1007/s002090100297}
}

@article{Sommen1997,
  author  = {Sommen, Frank},
  title   = {An Algebra of Abstract Vector Variables},
  journal = {Portugaliae Mathematica},
  volume  = {54},
  number  = {3},
  pages   = {287--310},
  year    = {1997},
  url     = {https://eudml.org/doc/47853}
}

@article{ZimmermannBernsteinSchneider2025,
  author  = {Zimmermann, Martha Lina and Bernstein, Swanhild and Schneider, Baruch},
  title   = {General Aspects of {Jackson} Calculus in {Clifford} Analysis},
  journal = {Advances in Applied Clifford Algebras},
  volume  = {35},
  pages   = {14},
  year    = {2025},
  doi     = {10.1007/s00006-025-01374-x}
}

\end{document}